\documentclass[12pt]{amsart}
\usepackage{}

\usepackage{amsmath}
\usepackage{amsfonts}
\usepackage{amssymb}
\usepackage[all]{xy}           

\usepackage{bbding}
\usepackage{txfonts}
\usepackage{amscd}

\usepackage[shortlabels]{enumitem}
\usepackage{ifpdf}
\ifpdf
  \usepackage[colorlinks,final,backref=page,hyperindex]{hyperref}
\else
  \usepackage[colorlinks,final,backref=page,hyperindex,hypertex]{hyperref}
\fi
\usepackage{tikz}
\usepackage[active]{srcltx}

\makeatletter

\newtheorem{thm}{Theorem}[section]

\newtheorem{pro}[thm]{Proposition}
\newtheorem{ex}[thm]{Example}
\newtheorem{rmk}[thm]{Remark}
\newtheorem{defi}[thm]{Definition}

\newcommand {\emptycomment}[1]{}

\newcommand{\lon }{\,\rightarrow\,}
\newcommand{\be }{\begin{equation}}
\newcommand{\ee }{\end{equation}}

\newcommand{\g}{\mathfrak g}
\newcommand{\h}{\mathfrak h}

\newcommand{\huaS}{\mathcal{S}}

\newcommand{\huaF}{\mathcal{F}}
\newcommand{\huaG}{\mathcal{G}}

\newcommand{\huaV}{\mathcal{V}}

\newcommand{\huaD}{\mathcal{D}}

\newcommand{\huaH}{\mathcal{H}}

\newcommand{\huaO}{{\mathcal{O}}}
\newcommand{\huaT}{\mathcal{T}}

\newcommand{\frkl}{\mathfrak l}

\newcommand{\frks}{\mathfrak s}

\newcommand{\frkC}{\mathfrak C}

\newcommand{\dM}{\mathrm{d}}
\newcommand{\RB}{\mathrm{RB}}

\newcommand{\Id}{{\rm{Id}}}

\newcommand{\br}[1]{   [ \cdot,    \cdot  ]   }

\newcommand{\CE}{\mathsf{CE}}

\newcommand{\Hom}{\mathrm{Hom}}

\newcommand{\gl}{\mathfrak {gl}}

\newcommand{\End}{\mathrm{End}}
\newcommand{\ad}{\mathrm{ad}}

\begin{document}

\title[Representations and cohomologies of relative Rota-Baxter Lie algebras]{Representations and cohomologies of relative Rota-Baxter Lie algebras and applications}

\author{Jun Jiang}
\address{Department of Mathematics, Jilin University, Changchun 130012, Jilin, China}
\email{jiangjmath@163.com}

\author{Yunhe Sheng}
\address{Department of Mathematics, Jilin University, Changchun 130012, Jilin, China}
\email{shengyh@jlu.edu.cn}

\vspace{-5mm}


\begin{abstract}
In this paper, first we give the notion of a representation of a relative Rota-Baxter Lie algebra and  introduce the cohomologies of a relative Rota-Baxter Lie algebra with coefficients in a representation. Then we classify abelian extensions of relative Rota-Baxter Lie algebras using the second cohomology group, and   classify   skeletal relative Rota-Baxter Lie 2-algebras using the third cohomology group as applications. At last, using the established general framework of representations and cohomologies of relative Rota-Baxter Lie algebras, we  give  the notion of representations of Rota-Baxter Lie algebras, which is  consistent with representations of Rota-Baxter associative algebras in the literature, and introduce the cohomologies of Rota-Baxter Lie algebras with coefficients in a representation.  Applications are also given to classify abelian extensions of Rota-Baxter Lie algebras and skeletal Rota-Baxter Lie 2-algebras.
\end{abstract}



\maketitle

\tableofcontents

\allowdisplaybreaks


\section{Introduction}

The purpose of this paper is to introduce the cohomology theory of a relative Rota-Baxter Lie algebra  with coefficients in a representation and give applications.

G. Baxter introduced the concept of Rota-Baxter operators on associative algebras
  in his study of
fluctuation theory in probability  \cite{Ba}. Recently it has found many
applications, including Connes-Kreimer's~\cite{CK} algebraic
approach to the renormalization in perturbative quantum field
theory. There are close connections between  Rota-Baxter operators and noncommutative
symmetric functions and Hopf algebras \cite{Fard,Gu0-1,Yu-Guo}.  Recently the relationship between Rota-Baxter operators and double Poisson algebras were studied in \cite{Goncharov}. For further details on
Rota-Baxter operators, see ~\cite{Gub-AMS,Gub}.
In the Lie algebra context, a Rota-Baxter operator was introduced independently in the 1980s as the
operator form of the classical Yang-Baxter equation that plays important roles in both mathematics and mathematical physics such as integrable
systems and quantum groups \cite{CP,STS}. B. A.
Kupershmidt introduced a more general notion, an  $\huaO$-operator on a Lie algebra (later also called
a relative Rota-Baxter operator or a generalized Rota-Baxter operator) in his study of the classical Yang-Baxter equation and
 related integrable systems \cite{Ku}. Relative Rota-Baxter operators provide solutions of the classical Yang-Baxter equation in the semidirect product Lie algebra and give rise to pre-Lie algebras \cite{Bai}.

 Recently the theories of deformations and homotopies of relative Rota-Baxter operators and relative Rota-Baxter algebras (both Lie and associative) are well developed in \cite{Das,DasM,LST,TBGS}. Among these studies, a cohomology theory was given to classify infinitesimal deformations and characterize the extendability of order $n$ deformations to order $n+1$ deformations. Let us briefly recall the approach taken in these studies. The first step is to construct the `controlling algebra' of the algebraic structure under consideration. It is well known that the controlling algebras of associative algebras and Lie algebras are graded Lie algebras given by the Gerstenhaber bracket and the Nijenhuis-Richardson bracket respectively \cite{Ge0,Ge,NR,NR2}. In \cite{LST}, using higher derived brackets \cite{Vo}, the authors constructed the controlling algebra of   relative Rota-Baxter Lie algebras, which turns out to be an $L_\infty$-algebra. More precisely, an $L_\infty$-algebra was given whose Maurer-Cartan elements are relative Rota-Baxter Lie algebra structures. Then using the controlling algebra, one can give a cohomology theory that can be applied to study infinitesimal  deformations and formal deformations.

 In general a cohomology theory is related to the representation theory. However, the cohomology theory developed in the above literatures was not based on representations of relative Rota-Baxter Lie algebras. Thus it is desirable to give the notion of representations of relative Rota-Baxter Lie algebras and develop  the corresponding cohomology theory that reduces to the cohomology given in \cite{LST} when the representation is the adjoint representation. This is the purpose of this paper. It is well known that a relative Rota-Baxter operator induces a pre-Lie algebra. We generalize this correspondence to the context of representations and cohomologies. We also give applications of the cohomologies of a relative Rota-Baxter Lie algebra with coefficients in a representation. More precisely, we use the second cohomology group to classify abelian extensions of relative Rota-Baxter Lie algebras, and use the third cohomology group to classify relative Rota-Baxter Lie 2-algebras, which was introduced in \cite{Sheng} under the terminology of $\huaO$-operators on Lie 2-algebras in the study of solutions of 2-graded Yang-Baxter equations. Note that in the associative algebra context, the notion of a representation (module) over a Rota-Baxter associative algebra  was given in \cite{GouLin}, and further studied in \cite{LQ,QGG,QP,ZGZ}. The representation of a relative Rota-Baxter Lie algebra introduced in this paper is consistent with the representation of Rota-Baxter associative algebras, namely a representation of a Rota-Baxter associative algebra naturally gives rise to a representation of the corresponding Rota-Baxter Lie algebra.

The paper is organized as follows. In Section \ref{sec:rc}, first we give the notion of a representation of a relative Rota-Baxter Lie algebra and show that a representation can be characterized by the semidirect product relative Rota-Baxter Lie algebra. We also construct the dual representation. We also show that a representation of a relative Rota-Baxter Lie algebra induces a representation of the underlying pre-Lie algebra.  In Section \ref{sec:cohomology}, we introduce the cohomology theory of a relative Rota-Baxter Lie algebra with coefficients in a representation. 1-cocycles can be used to characterize derivations on a relative Rota-Baxter Lie algebra. We establish the relation between the cohomology of a relative Rota-Baxter operator and the cohomology of the underlying pre-Lie algebra.  In Section \ref{sec:a},  we introduce the notion of abelian extensions of relative Rota-Baxter Lie algebras and show that there is a one-to-one correspondence between equivalent classes of abelian extensions and the second cohomology group. In Section \ref{sec:2},  first we recall the notion of relative Rota-Baxter Lie 2-algebras, and then we show that skeletal relative Rota-Baxter Lie 2-algebras are classified by the third cohomology group. Finally in Section \ref{sec:RB}, we introduce representations and cohomologies of Rota-Baxter Lie algebras using the above general framework of relative Rota-Baxter Lie algebras and give various applications.



\vspace{2mm}
\noindent
{\bf Acknowledgements. } This research is supported by NSFC (11922110).

\section{Representations  of relative Rota-Baxter Lie algebras}\label{sec:rc}

In this section, we introduce the notion of a representation of a relative Rota-Baxter Lie algebra. We give the semidirect products characterization of representations of relative Rota-Baxter Lie algebras. The dual representation is given.  A representation of a relative Rota-Baxter Lie algebra also induces a representation of the underlying pre-Lie algebra.

\begin{defi}
\begin{enumerate}
\item[\rm(i)] Let $(\g,[\cdot,\cdot]_\g)$ be a Lie algebra.
A linear operator $T:\g\longrightarrow \g$ is called a {\bf Rota-Baxter operator } if
\begin{equation*}
 [T(x),T(y)]_\g=T\big([T(x),y]_\g+ [x,T(y)]_\g \big), \quad \forall x, y \in \g.
\end{equation*}
Moreover, a Lie algebra $(\g,[\cdot,\cdot]_\g)$ with a Rota-Baxter operator $T$ is
called a {\bf Rota-Baxter Lie algebra}. We denote it by $(\g,[\cdot,\cdot]_\g,T)$.
\item[\rm(ii)]
A {\bf relative Rota-Baxter Lie algebra} is a triple $((\g,[\cdot,\cdot]_{\g}),(V,\rho),T)$, where $(\g,[\cdot,\cdot]_{\g})$ is a Lie algebra, $\rho:\g\rightarrow \gl(V)$ is a representation of $\g$ on a vector space $V$ and $T:V\rightarrow\g$ is a {\bf relative Rota-Baxter operator}, i.e.
\begin{equation*}
[Tu,Tv]_{\g}=T(\rho(Tu)(v)-\rho(Tv)(u)),\quad \forall u,v\in V.
\end{equation*}
\end{enumerate}
\end{defi}
Note that a Rota-Baxter operator on a Lie algebra is a relative Rota-Baxter operator with respect to the adjoint representation.
\begin{defi}
Let $((\g,[\cdot,\cdot]_{\g}),(V,\rho),T)$ and $((\g',[\cdot,\cdot]_{\g'}),(V',\rho'),T')$ be two relative Rota-Baxter Lie algebras. A homomorphism from $((\g,[\cdot,\cdot]_{\g}),(V,\rho),T)$ to $((\g',[\cdot,\cdot]_{\g'}),(V',\rho'),T')$ consists of a Lie algebra homomorphism $\phi:\g\rightarrow\g'$ and a linear map $\varphi: V\rightarrow V'$ such that
\begin{eqnarray*}
T'\circ\varphi&=&\phi\circ T,\\
\varphi\rho(x)(u)&=&\rho'(\phi(x))(\varphi(u)),\quad \forall x\in\g, u\in V.
\end{eqnarray*}
\end{defi}

It is well known that a classical skew-symmetric $r$-matrix $r\in\wedge^2\g$ gives rise to a relative Rota-Baxter operator $r^\sharp:\g^*\lon\g$ on $\g$ with respect to the coadjoint representation. We give another interesting example.

\begin{ex}{\rm
Let $W\stackrel{\huaT}\rightarrow\h$ be a complex of vector spaces. We define
\begin{equation*}
\End(W\stackrel{\huaT}\rightarrow\h)=\{(A_{0},A_{1})|A_{0}\in\gl(\h),A_{1}\in\gl(W), A_{0}\circ\huaT=\huaT\circ A_{1}\}.
\end{equation*}
Then
$((\End(W\stackrel{\huaT}\rightarrow\h), [\cdot,\cdot]), (\Hom(\h,\ker\huaT),\varsigma),\Omega)$
 is a relative Rota-Baxter Lie algebra, where the Lie bracket $[\cdot,\cdot]$ on  $\End(W\stackrel{\huaT}\rightarrow\h)$ is given by
\begin{equation*}
[(A_{0},A_{1}),(B_{0},B_{1})]=([A_{0},B_{0}],[A_{1},B_{1}]), \quad \forall (A_{0},A_{1}),(B_{0},B_{1})\in \End(W\stackrel{\huaT}\rightarrow\h),
\end{equation*}
the representation $\varsigma$ of the Lie algebra $(\End(W\stackrel{\huaT}\rightarrow\h), [\cdot,\cdot])$ on $\Hom(\h,\ker\huaT)$ is given by
\begin{eqnarray*}
\varsigma(A_{0},A_{1})(\Phi)&=&A_{1}\circ\Phi-\Phi\circ A_{0}, \quad\forall \Phi\in\Hom(\h,\ker\huaT),
\end{eqnarray*}
and the relative Rota-Baxter operator $\Omega:\Hom(\h,\ker\huaT)\lon \End(W\stackrel{\huaT}\rightarrow\h)$ is given by
\begin{eqnarray*}
\Omega(\Phi)&=&(0, \Phi\circ\huaT).
\end{eqnarray*}
}
\end{ex}

Now we introduce the notion of a representation of a relative Rota-Baxter Lie algebra.
\begin{defi}
A {\bf representation of a relative Rota-Baxter Lie algebra} $((\g,[\cdot,\cdot]_{\g}),(V,\rho), T)$ on a 2-term complex of vector spaces $W\stackrel{\huaT}\rightarrow \h$ is a triple $(\rho_{\h},\rho_{W},\mu)$, where $\rho_{\h}:\g\rightarrow\gl(\h)$ and $\rho_{W}:\g\rightarrow\gl(W)$ are representations of the Lie algebra $\g$ on vector spaces $\h$ and $W$ respectively, and $\mu: V\rightarrow\Hom(\h,W)$ is a linear map such that
the following equations are satisfied:
\begin{eqnarray}
\label{rep3}\mu(\rho(x)u)&=&\rho_{W}(x)\mu(u)-\mu(u)\rho_{\h}(x),\\
\label{rep4}\rho_{\h}(T(u))\circ \huaT&=&\huaT\circ\rho_{W}(T(u))+\huaT\circ\mu(u)\circ\huaT,
\end{eqnarray}
for all $x\in\g, u\in V.$
\end{defi}

We denote a representation by $[W\stackrel{\huaT}\rightarrow\h,\rho_{\h},\rho_{W},\mu]$.

\begin{ex}\label{ex:adjrep}{\rm
Let $((\g,[\cdot,\cdot]_{\g}),(V,\rho),T)$  be a relative Rota-Baxter Lie algebra. Define $\bar{\rho}:V\lon\Hom(\g,V)$ by
$$
\bar{\rho}(u)(x)=-\rho(x)u.
$$
Then it is straightforward to see that $[V\stackrel{T}\rightarrow\g,\ad,\rho,\bar{\rho}]$ is a representation, which is called the {\bf adjoint representation of the relative Rota-Baxter Lie algebra} $((\g,[\cdot,\cdot]_{\g}),(V,\rho), T)$.
}
\end{ex}


\begin{pro}\label{semiRB}
Let $[W\stackrel{\huaT}\rightarrow\h,\rho_{\h},\rho_{W},\mu]$ be a representation of a relative Rota-Baxter Lie algebra $((\g,[\cdot,\cdot]_{\g}),(V,\rho),T)$. Then $((\g\oplus\h, [\cdot,\cdot]_{\ltimes}),(V\oplus W, \varrho),\frak{T})$ is a relative Rota-Baxter Lie algebra, where $[\cdot,\cdot]_{\ltimes}$ is the semidirect product Lie bracket given by
\begin{equation*}
[x+\alpha,y+\beta]_{\ltimes}=[x,y]_\g+\rho_{\h}(x)\beta-\rho_{\h}(y)\alpha,\quad \forall x, y\in\g, \alpha,\beta\in\h,
\end{equation*}
the representation $\varrho: \g\oplus\h\rightarrow\gl(V\oplus W)$ is given by
\begin{eqnarray*}
\varrho(x+\alpha)(u+\xi)&=&\rho(x)u+\rho_{W}(x)\xi-\mu(u)\alpha,\quad \forall x\in\g,\alpha \in\h, u\in V, \xi\in W,
\end{eqnarray*}
and the relative Rota-Baxter operator $\frak{T}: V\oplus W\rightarrow\g\oplus\h$ is given by
\begin{eqnarray*}
\frak{T}(u+\xi)&=&T(u)+\huaT(\xi).
\end{eqnarray*}

\end{pro}

This relative Rota Baxter Lie algebra is called the {\bf semidirect product} of $((\g,[\cdot,\cdot]_{\g}),(V,\rho),T)$ and the representation $[W\stackrel{\huaT}\rightarrow\h,\rho_{\h},\rho_{W},\mu]$.
\begin{proof}
Since $\rho_\h$ is a representation of $\g$ on $\h$, it is obvious that  $(\g\oplus\h, [\cdot,\cdot]_{\ltimes})$ is a Lie algebra.

 By (\ref{rep3}), we have
\begin{eqnarray*}
[\varrho(x+\alpha),\varrho(y+\beta)](u+\xi)&=&\varrho(x+\alpha)\varrho(y+\beta)(u+\xi)-\varrho(y+\beta)\varrho(x+\alpha)(u+\xi)\\
&=&\varrho(x+\alpha)(\rho(y)u+\rho_{W}(y)\xi-\mu(u)\beta)\\&&-\varrho(y+\beta)(\rho(x)u
+\rho_{W}(x)\xi-\mu(u)\alpha)\\
&=&\rho(x)\rho(y)u+\rho_{W}(x)\rho_{W}(y)\xi-\rho_{W}(x)\mu(u)\beta-\mu(\rho(y)u)\alpha\\
&&-\rho(y)\rho(x)u-\rho_{W}(y)\rho_{W}(x)\xi+\rho_{W}(y)\mu(u)\alpha+\mu(\rho(x)u)\beta\\
&=&[\rho(x),\rho(y)]u+[\rho_{W}(x),\rho_{W}(y)]\xi\\
&&+\rho_{W}(y)\mu(u)\alpha-\mu(\rho(y)u)\alpha+\mu(\rho(x)u)\beta-\rho_{W}(x)\mu(u)\beta\\
&=&[\rho(x),\rho(y)]u+[\rho_{W}(x),\rho_{W}(y)]\xi+\mu(u)\rho_{\h}(y)\alpha-\mu(u)\rho_{\h}(x)\beta.
\end{eqnarray*}
Moreover,
\begin{eqnarray*}
\varrho([x+\alpha,y+\beta]_\ltimes)(u+\xi)&=&\varrho([x,y]_\g+\rho_{\h}(x)\beta-\rho_{\h}(y)\alpha)(u+\xi)\\
&=&\rho([x,y]_\g)u+\rho_{W}([x,y]_\g)\xi+\mu(u)\rho_{\h}(y)\alpha-\mu(u)\rho_{\h}(x)\beta.
\end{eqnarray*}
Since $\rho$ and $\rho_{W}$ are representations of $\g$ on $V$ and $W$ respectively, we have
\begin{equation*}
\varrho([x+\alpha,y+\beta]_\ltimes)(u+\xi)=[\varrho(x+\alpha),\varrho(y+\beta)](u+\xi),
\end{equation*}
which implies that $\varrho$ is a representation of the Lie algebra $(\g\oplus\h, [\cdot,\cdot]_\ltimes)$ on $V\oplus W$.

By (\ref{rep4}), we have
\begin{eqnarray*}
&&\frak{T}(\varrho(\frak{T}(u_{1}+\xi_{1}))(u_{2}+\xi_{2}))-\frak{T}(\varrho(\frak{T}(u_{2}+\xi_{2}))(u_{1}+\xi_{1}))\\
&=&(T+\huaT)(\varrho(T(u_{1})+\huaT(\xi_{1}))(u_{2}+\xi_{2}))-(T+\huaT)(\varrho(T(u_{2})+\huaT(\xi_{2}))(u_{1}+\xi_{1}))\\
&=&(T+\huaT)(\rho(T(u_{1}))u_{2}+\rho_{W}(T(u_{1}))\xi_{2}-\mu(u_{2})\huaT(\xi_{1}))\\
&&-(T+\huaT)(\rho(T(u_{2}))u_{1}+\rho_{W}(T(u_{2}))\xi_{1}-\mu(u_{1})\huaT(\xi_{2}))\\
&=&T(\rho(T(u_{1}))u_{2})+\huaT(\rho_{W}(T(u_{1}))\xi_{2}-\mu(u_{2})\huaT(\xi_{1}))\\
&&-T(\rho(T(u_{2}))u_{1})-\huaT(\rho_{W}(T(u_{2}))\xi_{1}-\mu(u_{1})\huaT(\xi_{2}))\\
&=&T(\rho(T(u_{1}))u_{2})-T(\rho(T(u_{2}))u_{1})+\huaT(\rho_{W}(T(u_{1}))\xi_{2})+\huaT(\mu(u_{1})\huaT(\xi_{2}))\\
&&-\huaT(\mu(u_{2})\huaT(\xi_{1}))-\huaT(\rho_{W}(T(u_{2}))\xi_{1})\\
&=&T(\rho(T(u_{1}))u_{2})-T(\rho(T(u_{2}))u_{1})+\rho_{\h}(T(u_{1}))\huaT(\xi_{2})-\rho_{\h}(T(u_{2}))\huaT(\xi_{1}).
\end{eqnarray*}
Moreover,
\begin{eqnarray*}
[\frak{T}(u_{1}+\xi_{1}),\frak{T}(u_{2}+\xi_{2})]_\ltimes&=&[T(u_{1})+\huaT(\xi_{1}),T(u_{2})+\huaT(\xi_{2})]_\ltimes\\
&=&[T(u_{1}),T(u_{2})]_\g+\rho_{\h}(T(u_{1}))\huaT(\xi_{2})-\rho_{\h}(T(u_{2}))\huaT(\xi_{1}).
\end{eqnarray*}
Since $T$ is a relative Rota-Baxter operator, we have
\begin{equation*}
[\frak{T}(u_{1}+\xi_{1}),\frak{T}(u_{2}+\xi_{2})]_\ltimes=\frak{T}(\varrho(\frak{T}(u_{1}+\xi_{1}))(u_{2}+\xi_{2}))-\frak{T}(\varrho(\frak{T}(u_{2}+\xi_{2}))(u_{1}+\xi_{1})),
\end{equation*}
which implies that $\frak{T}$ is a relative Rota-Baxter operator on the Lie algebra $(\g\oplus\h, [\cdot,\cdot]_\ltimes)$ with respect the representation $(V\oplus W, \varrho)$.
 Thus $((\g\oplus\h, [\cdot,\cdot]_{\ltimes}),(V\oplus W, \varrho),\frak{T})$ is a relative Rota-Baxter Lie algebra.
\end{proof}

Next we consider the {\bf dual representation} of a relative Rota-Baxter Lie algebra. Let   $[W\stackrel{\huaT}\rightarrow\h,\rho_{\h},\rho_{W},\mu]$ be a representation of a relative Rota-Baxter Lie algebra $((\g,[\cdot,\cdot]_{\g}),(V,\rho),T)$. Define  $\rho_\h^*$ and $\rho_W^*$ by
$$
\langle\rho^*_\h(x)\zeta, \alpha\rangle=-\langle \zeta, \rho_\h(x)\alpha\rangle, \quad \langle\rho^*_W(x)\epsilon, \xi\rangle=-\langle \epsilon, \rho_W(x)\xi\rangle,
$$
for all $x\in\g, \alpha\in\h, \xi\in W, \zeta\in\h^*, \epsilon\in W^*$. Then it is well known that  $\rho_\h^*$ and $\rho_W^*$ are representations of $\g$ on $\h^*$ and $W^*$ respectively.  We consider the complex $\h^*\stackrel{-\huaT^*}\rightarrow W^*$, where $\huaT^*$ is the usual dual of $\huaT$, i.e. $\langle\huaT^*(\zeta),\xi\rangle=\langle\zeta,\huaT(\xi) \rangle$. Define $\mu^*: V\rightarrow\Hom(W^*,\h^*)$ by
\begin{equation*}
 \langle \mu^*(u)  \epsilon, \alpha \rangle=\langle \epsilon, -\mu(u) \alpha \rangle,\quad \forall u\in V,\epsilon\in W^*, \alpha \in\h.
\end{equation*}

\begin{pro}\label{pro:dualrep}
Let $[W\stackrel{\huaT}\rightarrow\h,\rho_{\h},\rho_{W},\mu]$ be a representation of a relative Rota-Baxter Lie algebra $((\g,[\cdot,\cdot]_{\g}),(V,\rho),T)$. Then $[\h^*\stackrel{-\huaT^*}\rightarrow W^*, \rho_{W}^*, \rho_{\h}^*, \mu^*]$ is a representation of  $((\g,[\cdot,\cdot]_{\g}),(V,\rho),T)$, called the {\bf dual representation}.
\end{pro}
\begin{proof} Since  $\rho_\h^*$ and $\rho_W^*$ are representations of $\g$ on $\h^*$ and $W^*$, we only need to prove the following two equalities:
\begin{eqnarray}\label{dual 4}
\mu^*(\rho(x)u)&=&\rho_{\h}^*(x)\mu^*(u)-\mu^*(u)\rho_{W}^*(x),\quad \forall x\in\g, u\in V,\\
\label{dual 5}
\rho_{W}^*(T(u))\circ(-\huaT^*)&=&(-\huaT^*)\circ\rho_{\h}^*(T(u))+(-\huaT^*)\circ\mu^*(u)\circ(-\huaT^*).
\end{eqnarray}

By \eqref{rep3}, for all $\epsilon\in W^*, \alpha\in\h,$ we have
\begin{eqnarray*}
\langle\mu^*(\rho(x)u)\epsilon, ~\alpha\rangle&=&-\langle\epsilon, ~\mu(\rho(x)u)\alpha\rangle\\
&=&-\langle\epsilon, ~\rho_{W}(x)\mu(u)\alpha-\mu(u)\rho_{\h}(x)\alpha\rangle\\
&=&\langle\rho_{W}^*(x)\epsilon, ~\mu(u)\alpha\rangle-\langle\mu^*(u)\epsilon, ~\rho_{\h}(x)\alpha\rangle\\
&=&-\langle\mu^*(u)\rho_{W}^*(x)\epsilon, ~\alpha\rangle+\langle\rho_{\h}^*(x)\mu^*(u)\epsilon, ~\alpha\rangle\\
&=&\langle\rho_{\h}^*(x)\mu^*(u)\epsilon-\mu^*(u)\rho_{W}^*(x)\epsilon, ~\alpha\rangle,
\end{eqnarray*}
which implies that \eqref{dual 4} holds.

By \eqref{rep4}, for all $\xi\in W, \zeta\in \h^*,$ we have
\begin{eqnarray*}
\langle\rho_{W}^*(T(u))(-\huaT^*)(\zeta),~\xi\rangle&=&\langle\huaT^*(\zeta), ~\rho_W(T(u))\xi\rangle\\
&=&\langle\zeta, ~\huaT(\rho_W(T(u))\xi)\rangle\\
&=&\langle\zeta, ~\rho_{\h}(T(u))\huaT(\xi)-\huaT(\mu(u)\huaT(\xi))\rangle\\
&=&-\langle\rho_{\h}^*(T(u))\zeta, ~\huaT(\xi)\rangle-\langle\huaT^*(\zeta), ~\mu(u)(\huaT(\xi))\rangle\\
&=&\langle(-\huaT^*)(\rho_{\h}^*(T(u))(\zeta)), ~\xi\rangle+\langle(-\huaT^*)(\mu^*(u)(-\huaT^*)(\zeta)), ~\xi\rangle,
\end{eqnarray*}
which implies that \eqref{dual 5} holds.

Therefore, $[\h^*\stackrel{-\huaT^*}\rightarrow W^*, \rho_{W}^*, \rho_{\h}^*, \mu^*]$ is a representation of  $((\g,[\cdot,\cdot]_{\g}),(V,\rho),T)$.
\end{proof}

\begin{ex}{\rm
Let $((\g,[\cdot,\cdot]_{\g}),(V,\rho),T)$  be a relative Rota-Baxter Lie algebra.
Then by Example \ref{ex:adjrep} and Proposition \ref{pro:dualrep},  $[\g^*\stackrel{-T^*}\rightarrow V^*,\rho^*,\ad^*,\bar{\rho}^*]$ is a representation, which is called the {\bf coadjoint representation of the relative Rota-Baxter Lie algebra} $((\g,[\cdot,\cdot]_{\g}),(V,\rho), T)$.
}
\end{ex}


At the end of this section, we establish the relationship between representations of relative Rota-Baxter Lie algebras and representations of the underlying pre-Lie algebras. Recall that a  {\bf pre-Lie algebra} is a pair $(A,\cdot_A)$, where $A$ is a vector space and  $\cdot_A:A\otimes A\longrightarrow A$ is a bilinear multiplication
satisfying that for all $x,y,z\in A$,
$$(x\cdot_A y)\cdot_A z-x\cdot_A(y\cdot_A z)=(y\cdot_A x)\cdot_A
z-y\cdot_A(x\cdot_A z).$$
A pre-Lie algebra $(A,\cdot_A)$ gives rise to a Lie algebra  $(A,[\cdot,\cdot]_C)$, where the Lie bracket $[\cdot,\cdot]_C$ is the commutator of the pre-Lie multiplication. A {\bf  representation} of a pre-Lie algebra $(A,\cdot_A)$ on a vector
space $W$ consists of a pair
$(\theta,\vartheta)$, where $\theta:A\longrightarrow \gl(W)$ is a representation
of the Lie algebra $(A,[\cdot,\cdot]_C)$ on $W $ and $\vartheta:A\longrightarrow \gl(W)$ is a linear
map satisfying
\begin{eqnarray}\label{representation condition 2}
 \theta(x)\vartheta(y)u-\vartheta(y)\theta(x)u=\vartheta(x\cdot_A y)u-\vartheta(y)\vartheta(x)u, \quad \forall x,y\in A,~ u\in W.
\end{eqnarray}
It is well known that a relative Rota-Baxter operator $T:V\lon\g$ induces a pre-Lie algebra structure $\cdot_T$ on $V$ by $$u\cdot_T v=\rho(Tu)v.$$ See \cite{Bai} for more details. Now we generalize this correspondence to representations.

\begin{pro}
  Let $[W\stackrel{\huaT}\rightarrow\h,\rho_{\h},\rho_{W},\mu]$ be a representation of a relative Rota-Baxter Lie algebra $((\g,[\cdot,\cdot]_{\g}),(V,\rho),T)$. Then $(W; \theta, \vartheta)$ is a representation of the pre-Lie algebra $(V,\cdot_T),$ where $$\theta=\rho_{W}\circ T,\quad\vartheta(u)\xi=-\mu(u)\huaT(\xi),\quad \forall u\in V,\xi\in W.$$
\end{pro}
\begin{proof}
Since $T$ is relative Rota-Baxter operator, for any $u,v\in V$, we have
\begin{equation}\label{prerep1}
\theta([u,v]_C)=\rho_{W}(T(\rho(T(u))v-\rho(T(v)u)))=\rho_{W}([T(u),T(v)]_\g)=[\theta(u),\theta(v)],
\end{equation}
which implies that $\theta$ is a representation of Lie algebra $(V,[\cdot,\cdot]_C)$ on $W$.

By \eqref{rep3} and \eqref{rep4}, for any $\xi\in W$, we have
\begin{eqnarray*}
 \nonumber\theta(u)\vartheta(v)\xi-\vartheta(v)\theta(u)\xi
&=&\nonumber-\theta(u)\mu(v)\huaT(\xi)-\vartheta(v)\rho_W(T(u))\xi\\
&=&\nonumber-\rho_{W}(T(u))\mu(v)\huaT(\xi)+\mu(v)\huaT\rho_W(T(u))\xi\\
&=&\nonumber-\mu(v)\rho_{\h}(T(u))\huaT(\xi)-\mu(\rho(T(u))v)\huaT\xi+\mu(v)\huaT\rho_W(T(u))\xi\\
&=&\nonumber\vartheta(u\cdot_T v)\xi-\mu(v)\huaT\mu(u)\huaT\xi\\
&=&\label{prerep2}\vartheta(u\cdot_T v)\xi-\vartheta(v)\vartheta(u)\xi.
\end{eqnarray*}
Thus   $(W; \theta, \vartheta)$ is a representation of the pre-Lie algebra $(V,\cdot_T).$
\end{proof}

\section{Cohomology of relative Rota-Baxter Lie algebras}\label{sec:cohomology}
In this section, we introduce a cohomology theory of relative Rota-Baxter Lie algebras with coefficients in  arbitrary representations. In particular, we introduce the notion of derivations on a relative Rota-Baxter Lie algebra which can be characterized by 1-cocycles. We also establish the relationship between the cohomology of a relative Rota-Baxter operator and the cohomology of the underlying pre-Lie algebra with coefficients in a representation.

First we recall the cohomology of a relative Rota-Baxter Lie algebra with coefficients in the adjoint representation given in \cite{LST}.
Let $((\g,[\cdot,\cdot]_{\g}),(V,\rho),T)$ be a relative Rota-Baxter Lie algebra. Define the set of  $0$-cochains $\frkC^0(\g,\rho,T)$ to be $0$, and define
the set of  $1$-cochains $\frkC^1(\g,\rho,T)$ to be $\gl(\g)\oplus \gl(V)$. For $n\geq 2$, define the space of    $n$-cochains $\frkC^n(\g,\rho,T)$ by
\begin{equation*}
\frkC^n(\g,\rho,T):=\Big(\Hom(\wedge^n\g,\g)\oplus \Hom(\wedge^{n-1}\g\otimes V,V)\Big)\oplus\Hom(\wedge^{n-1}V,\g).
\end{equation*}

Define the coboundary operator $\huaD:\frkC^n(\g,\rho,T)\rightarrow\frkC^{n+1}(\g,\rho,T)$ by
\begin{equation*}
\huaD(f,\theta)=(\delta f,\partial\theta+h_{T}f),\quad f\in\Hom(\wedge^n\g,\g)\oplus\Hom(\wedge^{n-1}\g\otimes V,V), \theta\in\Hom(\wedge^{n-1}V,\g),
\end{equation*}
where
\begin{itemize}
   \item $\delta: \Hom(\wedge^n\g,\g)\oplus\Hom(\wedge^{n-1}\g\otimes V, V)\rightarrow\Hom(\wedge^{n+1}\g,\g)\oplus\Hom(\wedge^n\g\otimes V, V)$ is defined as follows. For $f=(f_\g,f_V)\in\Hom(\wedge^n\g,\g)\oplus\Hom(\wedge^{n-1}\g\otimes V, V)$, we write
       \begin{equation*}
       \delta(f)=\big((\delta (f))_\g, (\delta (f))_V\big).
       \end{equation*}
       Then $(\delta (f))_\g=\dM_{\CE}f_{\g}$, where $\dM_{\CE}$ is the Chevalley-Eilenberg coboundary operator of the Lie algebra $\g$ with coefficients in the adjoint representation \cite{CHE}, and $(\delta(f))_V$ is given by
       \begin{eqnarray}
\nonumber&&(\delta(f))_V(x_1,\cdots,x_{n},v)\\
\label{cohomology-algebra-rep-V}&=&\sum_{1\le i<j\le n}(-1)^{i+j}f_V([x_i,x_j]_\g,x_1,\cdots,\hat{x}_i,\cdots,\hat{x}_j,\cdots,x_n,v)+(-1)^{n-1}\rho(f_\g(x_1,\cdots,x_{n}))v\\
\nonumber&&+\sum_{i=1}^{n}(-1)^{i+1}\Big(\rho(x_i)f_V(x_1,\cdots,\hat{x}_i,\cdots,x_n,v)-f_V\big(x_1,\cdots,\hat{x}_i,\cdots,x_n,\rho(x_i)v\big)\Big).
\end{eqnarray}
  \item $\partial: \Hom(\wedge^{n-1}V,\g)\rightarrow\Hom(\wedge^nV,\g)$ is defined by
  \begin{eqnarray*}
&& (\partial \theta)(v_1,\cdots,v_{n})\notag\\
&=&\sum_{i=1}^{n}(-1)^{i+1}[T(v_i),\theta(v_1,\cdots,\hat{v}_i,\cdots, v_{n})]_\g+\sum_{i=1}^{n}(-1)^{i+1}T(\rho(\theta(v_1,\cdots,\hat{v}_i,\cdots, v_{n}))v_i)\label{eq:odiff}\\
&&+\sum_{1\le i<j\le n}(-1)^{i+j}\theta(\rho(Tv_i)v_j-\rho(Tv_j)v_i,v_1,\cdots,\hat{v}_i,\cdots,\hat{v}_j,\cdots, v_{n}). \notag
\end{eqnarray*}
  \item $h_T: \Hom(\wedge^n\g,\g)\oplus\Hom(\wedge^{n-1}\g\otimes V,V)\rightarrow\Hom(\wedge^nV,\g)$ is defined by
  \begin{eqnarray*}\label{key-cohomology-T}
&&\nonumber( h_Tf)(v_1,\cdots,v_n)\\
&=&(-1)^{n}f_\g(T(v_1),\cdots,T(v_n))+\sum_{i=1}^{n}(-1)^{i+1}T\big(f_V(T(v_1),\cdots,T(v_{i-1}),T(v_{i+1}),\cdots,T(v_n),v_i)\big).
\end{eqnarray*}
\end{itemize}

\begin{thm}\label{1}{\rm (\cite{LST})}
  With the above notations,  $(\oplus _{n=0}^{+\infty}\frkC^n(\g,\rho,T),\huaD)$ is a cochain complex, i.e. $$\huaD\circ \huaD=0.$$
\end{thm}

Let $[W\stackrel{\huaT}\rightarrow\h, \rho_\h, \rho_W, \mu]$ be a representation of $((\g,[\cdot,\cdot]_{\g}),(V;\rho), T)$.  Define the set of $0$-cochains $\frkC^0(\g,\rho,T;\rho_{\h},\rho_{W},\mu)$ to be $0$, and define the set of $1$-cochains $\frkC^1(\g,\rho,T;\rho_\h,\rho_W,\mu)$ to be $\Hom(\g,\h)\oplus\Hom(V,W)$. For $n\geq 2$, we define the set of $n$-cochains $\frkC^n(\g,\rho,T;\rho_{\h},\rho_{W},\mu)$  by
$$
\frkC^n(\g,\rho,T;\rho_{\h},\rho_{W},\mu)=\Big(\Hom(\wedge^{n}\g,\h)\oplus \Hom(\wedge^{n-1}\g\otimes V,W)\Big)\oplus\Hom(\wedge^{n-1}V,\h).
$$
Define the coboundary operator
\begin{equation*}
\huaD_R:\frkC^n(\g,\rho,T;\rho_{\h},\rho_{W},\mu)\lon \frkC^{n+1}(\g,\rho,T;\rho_{\h},\rho_{W},\mu)
\end{equation*}
by
\begin{equation}\label{eq:Dexplicit}
\huaD_R(f,\theta)=(\delta f,\partial\theta+h_{T}(f)),
\end{equation}
for all $f\in\Hom(\wedge^{n}\g,\h)\oplus \Hom(\wedge^{n-1}\g\otimes V,W), \theta\in\Hom(\wedge^{n-1}V,\h)$,
 where
 \begin{itemize}
   \item  $\delta:\Hom(\wedge^{n}\g,\h)\oplus \Hom(\wedge^{n-1}\g\otimes V,W)\lon \Hom(\wedge^{n+1}\g,\h)\oplus \Hom(\wedge^{n}\g\otimes V,W)$ is defined as follows. For $f=(f_{\h},f_{W})\in\Hom(\wedge^{n}\g,\h)\oplus\Hom(\wedge^{n-1}\g\otimes V,W)$,   we write
\begin{equation*}
\delta(f)=\big((\delta(f))_{\h},(\delta(f))_{W}\big).
\end{equation*}
Then $(\delta(f))_{\h}=\dM_{\CE}f_\h$, where $\dM_\CE$ is the Chevalley-Eilenberg coboundary operator of the Lie algebra $\g$ with coefficients in the representation $(\h,\rho_\h)$, and $(\delta(f))_{W}$ is given by
\begin{eqnarray*}
&&(\delta(f))_{W}(x_{1},x_{2},\cdots,x_{n},v)\\
&=&\sum_{1\leq i<j\leq n}(-1)^{i+j}f_{W}([x_{i},x_{j}]_\g,x_{1},\cdots,\widehat{x_{i}},\cdots,\widehat{x_{j}},x_{n},v)-(-1)^{n-1}\mu(v)f_{\h}(x_{1},\cdots,x_{n})\\
&&+\sum_{i=1}^{n}(-1)^{i+1}\big(\rho_{W}(x_{i})f_{W}(x_{1},\cdots,\widehat{x_{i}},\cdots,x_{n},v)-f_{W}(x_{1},\cdots,\widehat{x_{i}},\cdots,x_{n},\rho(x_{i})v)\big).
\end{eqnarray*}
\item $\partial:\Hom(\wedge^{n-1}V,\h)\lon \Hom(\wedge^{n}V,\h)$ is defined by
\begin{eqnarray*}
&&\partial\theta(v_{1},\cdots,v_{n})\\
&=&\sum_{i=1}^{n}(-1)^{i+1}\rho_{\h}(T(v_{i}))\theta(v_{1},\cdots,\hat{v_{i}},\cdots,v_{n})-\sum_{i=1}^{n}(-1)^{i+1}\huaT(\mu(v_{i})\theta(v_{1},\cdots,\widehat{v_{i}},\cdots,v_{n}))\\
&&+\sum_{1\leq i<j\leq n}(-1)^{i+j}\theta(\rho(T(v_{i}))v_{j}-\rho(T(v_{j}))v_{i},v_{1},\cdots,\hat{v_{i}},\cdots,\hat{v_{j}},\cdots,v_{n}).
\end{eqnarray*}

\item $h_{T}: \Hom(\wedge^{n}\g,\h)\oplus\Hom(\wedge^{n-1}\g\otimes V,W)\rightarrow\Hom(\wedge^{n}V,\h)$ is defined by
\begin{eqnarray*}
&&h_{T}(f)(v_{1},\cdots,v_{n})\\
&=&(-1)^{n}f_{\h}(T(v_{1}),T(v_{2}),\cdots,T(v_{n}))+\sum_{i=1}^{n}(-1)^{i+1}\huaT (f_{W}(T(v_{1}),\cdots,T(v_{i-1}),T(v_{i+1}),T(v_{n}),v_{i})).
\end{eqnarray*}

 \end{itemize}

\begin{thm}\label{cohomology-of-LLT}
  With the above notations,  $(\oplus _{n=0}^{+\infty}\frkC^n(\g,\rho,T;\rho_{\h},\rho_{W},\mu),\huaD_R)$ is a cochain complex, i.e. $$\huaD_R\circ \huaD_R=0.$$
\end{thm}
\begin{proof}
 We only give a sketch of the proof and leave details to readers. Consider the semidirect product relative Rota-Baxter Lie algebra $ ((\g\oplus \h,[\cdot,\cdot]_{\ltimes}), (V\oplus W,\varrho),\frak{T})$ given in Proposition \ref{semiRB}, and the associated cochain complex $(\oplus_{n=0}^{+\infty}\frkC^n(\g\oplus\h,\varrho, \frak{T}),\huaD)$ given in Theorem \ref{1}. It is straightforward to deduce that
  $(\oplus _{n=0}^{+\infty}\frkC^n(\g,\rho,T;\rho_\h,\rho_W,\mu),\huaD_R)$ is a subcomplex of $(\oplus_{n=0}^{+\infty}\frkC^n(\g\oplus\h,\varrho, \frak{T}),\huaD)$. Thus, $\huaD_R\circ \huaD_R=0.$
\end{proof}
\begin{defi}
  The  cohomology of the cochain complex  $(\oplus_{n=0}^{+\infty}\frkC^n(\g,\rho,T;\rho_{\h},\rho_{W},\mu),\huaD_R)$ is called {\bf the  cohomology of the relative Rota-Baxter Lie algebra} with coefficients in the representation $[W\stackrel{\huaT}\rightarrow\h, \rho_\h, \rho_W, \mu]$. The corresponding $n$-th cohomology group is denoted by $\huaH^n(\g,\rho,T;\rho_{\h},\rho_{W},\mu)$.
\end{defi}

\begin{rmk}
  When the representation is the adjoint representation, the above cohomology reduces to the cohomology given in \cite{LST}.
\end{rmk}

Now we analyze  the first cohomology group $\huaH^1(\g, \rho, T; \rho_{\h}, \rho_{W}, \mu)$. Applications of the second and the third cohomology groups will be given in  the following sections. For all $f=(f_\h, f_W)\in\Hom(\g,\h)\oplus\Hom(V,W)=\frkC^1(\g,\rho,T;\rho_\h,\rho_W,\mu)$,   we have
\begin{eqnarray*}
(\delta(f_\h, f_W))_\h&=&\dM_{\CE}f_\h,\\
(\delta(f_\h, f_W))_W(x,v)&=& -\mu(v)f_\h(x)+\rho_W(x)f_W(v)-f_W(\rho(x)v),\\
h_{T}(f_\h, f_W)(v)&=&-f_\h(T(v))+\huaT f_W(v), \quad \forall x\in\g, v\in V.
\end{eqnarray*}
 Therefore, $(f_\h, f_W)$ is a 1-cocycle if and only if
 \begin{eqnarray}
  \label{eq:11}\dM_{\CE}f_\h&=&0,\\
  \label{eq:12} -\mu(v)f_\h(x)+\rho_W(x)f_W(v)-f_W(\rho(x)v)&=&0,\\
   \label{eq:13}-f_\h(T(v))+\huaT f_W(v)&=&0.
 \end{eqnarray}
Since $\frkC^0(\g,\rho,T;\rho_{\h},\rho_{W},\mu)=0$, there are no 1-coboundaries. Thus we have
\begin{equation*}
\huaH^1(\g, \rho, T; \rho_{\h}, \rho_{W}, \mu)=\{(f_\h, f_W)\in\Hom(\g,\h)\oplus\Hom(V,W) ~\mbox{satisfying}~\eqref{eq:11}\mbox{-}\eqref{eq:13}\}.
\end{equation*}

Derivations on a Lie algebra $\g$ are $1$-cocycles on $\g$ with coefficients in the adjoint representation. Follow this philosophy, we define derivations on a relative Rota-Baxter Lie algebra $((\g,[\cdot,\cdot]_{\g}),(V,\rho),T)$ as follows.

\begin{defi}\label{defi:der}
  A {\bf derivation on a relative Rota-Baxter Lie algebra} $((\g,[\cdot,\cdot]_{\g}),(V,\rho),T)$ consists of linear maps  $ f_\g \in\Hom(\g,\g)$ and $f_V\in\Hom(V,V)$  such that $f_\g$ is a derivation on  the Lie algebra $\g$ and the following equalities hold:
\begin{eqnarray}
\label{eq:der1}T\circ f_V&=&f_\g\circ T,\\
\label{eq:der2}f_V(\rho(x)u)&=&\rho(x)f_V(u)+\rho(f_\g(x))u,\quad \forall x\in\g, u\in V.
\end{eqnarray}
Or in terms of commutative diagrams:
 \[
\small{ \xymatrix{
 V\ar[d]^{T} \ar[rr]^{   f_V} &  &  V\ar[d]^{T}  \\
 \g\ar[rr]^{  f_\g } & &\g,}
}   \quad  \small{ \xymatrix{
 V \ar[rr]^{   f_V} &  &  V  \\
 \g\otimes V  \ar[u]^{\rho}\ar[rr]^{  f_\g\otimes \Id_V+\Id_\g\otimes f_V} & &\g\otimes V\ar[u]_{\rho}.}
}
\]

 A {\bf derivation on a Rota-Baxter Lie algebra} $(\g,[\cdot,\cdot]_{\g}T)$ is a linear map  $ f \in\Hom(\g,\g)$   such that $f$ is a derivation on  the Lie algebra $\g$ and $f\circ T=T\circ f$.
\end{defi}

By \eqref{eq:11}\mbox{-}\eqref{eq:13} and the definition of the adjoint representation of a relative Rota-Baxter Lie algebra, we immediately have the following characterization of derivations on a relative Rota-Baxter Lie algebra.
\begin{pro}\label{pro:der}
A  derivation on a relative Rota-Baxter Lie algebra $((\g,[\cdot,\cdot]_{\g}),(V,\rho),T)$ is a $1$-cocycle on $((\g,[\cdot,\cdot]_{\g}),(V,\rho),T)$  with coefficients in the adjoint representation.
\end{pro}

It is known that $\left(\begin{array}{cc}
  0&T\\0&0
\end{array}\right)$ is a Rota-Baxter operator on the semidirect product Lie algebra $\g\ltimes_\rho V$.
\begin{pro}
$(f_\g,f_V)$ is a derivation on a relative Rota-Baxter Lie algebra $((\g,[\cdot,\cdot]_{\g}),(V,\rho),T)$ if and only if $\left(\begin{array}{cc}
  f_\g&0\\0&f_V
\end{array}\right)$ is a derivation on the   Rota-Baxter Lie algebra $(\g\ltimes_\rho V,\left(\begin{array}{cc}
  0&T\\0&0
\end{array}\right))$.

\end{pro}

\begin{proof}
  It is straightforward to deduce that \eqref{eq:der2} and the fact that $f_\g$ is a derivation on the Lie algebra $\g$ is equivalent to that $\left(\begin{array}{cc}
  f_\g&0\\0&f_V
\end{array}\right)$ is a derivation on the semidirect product Lie algebra $\g\ltimes_\rho V$. Moreover, it is obvious that \eqref{eq:der1} is equivalent to \begin{equation}\label{eq:comm}
\left(\begin{array}{cc}
  f_\g&0\\0&f_V
\end{array}\right)\left(\begin{array}{cc}
  0&T\\0&0
\end{array}\right)=\left(\begin{array}{cc}
  0&T\\0&0
\end{array}\right)\left(\begin{array}{cc}
  f_\g&0\\0&f_V
\end{array}\right).
\end{equation}
The proof is finished.
\end{proof}

We can pick a subcomplex and a quotient complex from    $(\oplus_{n=0}^{+\infty}\frkC^n(\g,\rho,T;\rho_{\h},\rho_{W},\mu),\huaD_R)$.
For $n> 1$, denote by \begin{eqnarray*}
\bar{\frkC}^n(\g,\rho;\rho_{\h},\rho_{W})&:=&\Hom(\wedge^{n}\g,\h)\oplus \Hom(\wedge^{n-1}\g\otimes V,W),\\
\tilde{\frkC}^n(T;\rho_{\h},\mu)&:=& \Hom(\wedge^{n-1}V,\h).
\end{eqnarray*}
Define $\tilde{\frkC}^1(T;\rho_{\h},\mu)=\tilde{\frkC}^0(T;\rho_{\h},\mu)=0$ and $\bar{\frkC}^0(\g,\rho;\rho_{\h},\rho_{W}) =0,~\bar{\frkC}^1(\g,\rho;\rho_{\h},\rho_{W})=\Hom(\g,\h)\oplus \Hom(V,W).$ It is obvious that
$$
 \frkC^n(\g,\rho,T;\rho_{\h},\rho_W,\mu)=\bar{\frkC}^n(\g,\rho;\rho_{\h},\rho_{W})\oplus \tilde{\frkC}^n(T;\rho_{\h},\mu).
$$
Then by Theorem \ref{cohomology-of-LLT} and the definition of $\huaD_R$, $(\oplus_{n=0}^{+\infty}\tilde{\frkC}^n(T;\rho_{\h},\mu),\partial)$ is a subcomplex of the cochain complex  $(\oplus _{n=0}^{+\infty}\frkC^n(\g,\rho,T;\rho_\h,\rho_W,\mu),\huaD_R)$ and $(\oplus_{n=0}^{+\infty}\bar{\frkC}^n(\g,\rho;\rho_{\h},\rho_{W}),\delta)$ is a quotient cochain complex.
The formula of the coboundary operator $\huaD_R$ can be well-explained by the following diagram:
 \[
\small{ \xymatrix{
\cdots
\longrightarrow \bar{\frkC}^n(\g,\rho;\rho_{\h},\rho_{W})\ar[dr]^{h_T} \ar[r]^{\qquad\delta} & \bar{\frkC}^{n+1}(\g,\rho;\rho_{\h},\rho_{W}) \ar[dr]^{h_T} \ar[r]^{\delta\qquad}  & \bar{\frkC}^{n+2}(\g,\rho;\rho_{\h},\rho_{W})\longrightarrow\cdots  \\
\cdots\longrightarrow \tilde{\frkC}^n(T;\rho_{\h},\mu) \ar[r]^{\qquad\partial} &\tilde{\frkC}^{n+1}(T;\rho_{\h},\mu)\ar[r]^{\partial\qquad}&\tilde{\frkC}^{n+2}(T;\rho_{\h},\mu)\longrightarrow \cdots.}
}
\]
Denote the corresponding cohomology groups by $\huaH^n(\g,\rho;\rho_{\h},\rho_{W})$ and $\huaH^n(T;\rho_{\h},\mu)$ respectively. The former can be served as {\bf the cohomology groups of a LieRep pair} (a Lie algebra with a representation) with coefficients in an arbitrary representation, while the latter can be served as {\bf the cohomology groups of a relative Rota-Baxter operator} with coefficients in an arbitrary representation. See \cite{LST,TBGS} for more details about the cohomology groups of a LieRep pair and a relative Rota-Baxter operator with coefficients in themselves.

We have the following results connecting the various cohomology groups.

\begin{thm}\label{cohomology-exact}
Let $[W\stackrel{\huaT}\rightarrow\h, \rho_\h, \rho_W, \mu]$ be a representation of a relative Rota-Baxter Lie algebra $((\g,[\cdot,\cdot]_{\g}),(V;\rho), T)$.  Then there is a short exact sequence of the  cochain complexes:
$$
0\longrightarrow(\oplus_{n=0}^{+\infty}\tilde{\frkC}^n(T;\rho_{\h},\mu),\partial)\stackrel{\iota}{\longrightarrow}(\oplus _{n=0}^{+\infty} \frkC^n(\g,\rho,T;\rho_{\h},\rho_W,\mu),\huaD_R)\stackrel{p}{\longrightarrow} (\oplus_{n=0}^{+\infty}\bar{\frkC}^n(\g,\rho;\rho_{\h},\rho_{W}),\delta)\longrightarrow 0,
$$
where $\iota$ and $p$ are the inclusion map and the projection map.

Consequently, there is a long exact sequence of the  cohomology groups:
$$
\cdots\longrightarrow\huaH^n(T;\rho_{\h},\mu)\stackrel{\huaH^n(\iota)}{\longrightarrow}\huaH^n(\g,\rho,T;\rho_{\h},\rho_W,\mu)\stackrel{\huaH^n(p)}{\longrightarrow} \huaH^n(\g,\rho;\rho_{\h},\rho_{W})\stackrel{c^n}\longrightarrow \huaH^{n+1}(T;\rho_{\h},\mu)\longrightarrow\cdots,
$$
where the connecting map $c^n$ is defined by
$
c^n([\alpha])=[h_T\alpha],$  for all $[\alpha]\in \huaH^n(\g,\rho;\rho_{\h},\rho_{W}).$
\end{thm}
\begin{proof}
 By  \eqref{eq:Dexplicit}, we have the short exact sequence  of cochain complexes which induces a long exact sequence of cohomology groups.   Also by \eqref{eq:Dexplicit},   $c^n$ is given by
$
c^n([f])=[h_Tf],$ for all $f\in \bar{\frkC}^n(\g,\rho;\rho_{\h},\rho_{W})$ and $[f]$ denotes the cohomological class of $f$.
\end{proof}

At the end of this section we recall the cohomology of a pre-Lie algebra  and establish the relation between the cohomology of a relative Rota-Baxter operator (that is the cohomology of the cochain complex $(\oplus_{n=0}^{+\infty}\tilde{\frkC}^n(T;\rho_{\h},\mu),\partial)$) and the cohomology of the underlying pre-Lie algebra.

The cohomology for a pre-Lie algebra $(A,\cdot_A)$ with coefficients in a representation $(W;\theta, \vartheta)$ is given as follows \cite{Dz}.
The set of $n$-cochains  $C^n(A,V)$ is given by
$$C^n(A,V)=\Hom(\wedge^{n-1}A\otimes A,W),\
n\geq 1.$$  The coboundary operator $\dM:\Hom(\wedge^{n-1}A\otimes A,W)\longrightarrow \Hom(\wedge^{n}A\otimes A,W)$ is given by
 \begin{eqnarray}
 \nonumber(\dM f)(x_1, \cdots,x_{n+1})
 \nonumber&=&\sum_{i=1}^{n}(-1)^{i+1}\theta(x_i)f(x_1, \cdots,\hat{x_i},\cdots,x_{n+1})\\
 \nonumber&&+\sum_{i=1}^{n}(-1)^{i+1}\vartheta(x_{n+1})f(x_1, \cdots,\hat{x_i},\cdots,x_n,x_i)\\
 \nonumber&&-\sum_{i=1}^{n}(-1)^{i+1}f(x_1, \cdots,\hat{x_i},\cdots,x_n,x_i\cdot_A x_{n+1})\\
\nonumber &&+\sum_{1\leq i<j\leq n}(-1)^{i+j}f([x_i,x_j]_C,x_1,\cdots,\hat{x_i},\cdots,\hat{x_j},\cdots,x_{n+1}),
\end{eqnarray}
for all $f\in \Hom(\wedge^{n-1}A\otimes A,W)$, $x_i\in A,~i=1,\cdots,n+1$. The cohomology of the cochain complex $(\oplus_{n\geq 1}C^n(A,V),\dM )$ is taken to be {\bf the cohomology of the pre-Lie algebra} $(A,\cdot_A)$ with coefficients in the representation $(W;\theta, \vartheta)$. We denote the corresponding $n$-th cohomology group by
$H_{\rm preLie}^n(A,W)$ and
$H_{\rm preLie}^\bullet(A,W):=\oplus_{n}H_{\rm preLie}^n(A,W).$

Let $[W\stackrel{\huaT}\rightarrow\h, \rho_\h, \rho_W, \mu]$ be a representation of a relative Rota-Baxter Lie algebra $((\g,[\cdot,\cdot]_{\g}),(V;\rho), T)$. Define a linear map $\Xi: \Hom(\wedge^{n-1}V,\h)\rightarrow\Hom(\wedge^{n-1}V\otimes V, W)$ by
\begin{equation*}
\Xi(\omega)(v_1,v_2,\cdots,v_{n-1},v_n)=\mu(v_n)\omega(v_1,v_2,\cdots,v_{n-1}),\quad \forall \omega\in\Hom(\wedge^{n-1}V,\h).
\end{equation*}
Then we have
\begin{thm}
Let $[W\stackrel{\huaT}\rightarrow\h,\rho_{\h},\rho_{W},\mu]$ be a representation of a relative Rota-Baxter Lie algebra $((\g,[\cdot,\cdot]_{\g}),(V,\rho),T)$. Then $\Xi$ is a homomorphism from the cochain complex  $(\oplus_{n=0}^{+\infty}\tilde{\frkC}^n(T;\rho_{\h},\mu),\partial)$ to $(\oplus_{n\geq 1}C^n(A,V),\dM )$. That is, we have the following commutative diagram:
\[
\small{ \xymatrix{
\cdots
\longrightarrow \Hom(\wedge^{n-1}V,\h) \ar[d]^{\Xi} \ar[r]^{\qquad\partial} & \Hom(\wedge^nV, \h) \ar[d]^{\Xi} \ar[r]  & \cdots  \\
\cdots\longrightarrow \Hom(\wedge^{n-1}V\otimes V, W) \ar[r]^{\qquad \dM} &\Hom(\wedge^nV\otimes V, W)\ar[r]& \cdots.}
}
\]
Consequently, $\Xi$ induces a homomorphism $\Xi_{*}:\huaH^n(T;\rho_{\h},\mu)\rightarrow H_{\rm preLie}^n(V,W)$ between the corresponding cohomology groups.
\end{thm}
\begin{proof}
Indeed, for any $\omega\in \Hom(\wedge^{n-1}V,\h),$ we have
\begin{eqnarray*}
 \dM(\Xi(\omega))(u_1, \cdots,u_n,u_{n+1})
&=&\sum_{i=1}^{n}(-1)^{i+1}\theta(u_i)\Xi(\omega)(u_1, \cdots,\hat{u_i},\cdots,u_{n+1})\\
&&+\sum_{i=1}^{n}(-1)^{i+1}\vartheta(u_{n+1})\Xi(\omega)(u_1, \cdots,\hat{u_i},\cdots,u_n,u_i)\\
&&-\sum_{i=1}^{n}(-1)^{i+1}\Xi(\omega)(u_1, \cdots,\hat{u_i},\cdots,u_n,u_i\cdot_T u_{n+1})\\
&&+\sum_{1\leq i<j\leq n}(-1)^{i+j}\Xi(\omega)([u_i,u_j]_C,u_1,\cdots,\hat{u_i},\cdots,\hat{u_j},\cdots,u_{n+1})\\
&=&\sum_{i=1}^{n}(-1)^{i+1}\rho_{W}(Tu_i)\mu(u_{n+1})\omega(u_1, \cdots,\hat{u_i},\cdots,u_{n})\\
&&-\sum_{i=1}^{n}(-1)^{i+1}\mu(u_{n+1})\huaT\mu(u_i)\omega(u_1, \cdots,\hat{u_i},\cdots,u_n)\\
&&-\sum_{i=1}^{n}(-1)^{i+1}\mu(u_i\cdot_T u_{n+1})\omega(u_1, \cdots,\hat{u_i},\cdots,u_n)\\
&&+\sum_{1\leq i<j\leq n}(-1)^{i+j}\mu(u_{n+1})\omega([u_i,u_j]_C,u_1,\cdots,\hat{u_i},\cdots,\hat{u_j},\cdots,u_{n}),
\end{eqnarray*}
and
\begin{eqnarray*}
 \Xi(\partial\omega)(u_1,u_2,\cdots,u_n,u_{n+1})
&=&\sum_{i=1}^{n}(-1)^{i+1}\mu(u_{n+1})\rho_{\h}(T(u_i))\omega(u_1,\cdots,\hat{u_i},\cdots,u_n)\\
&&-\sum_{i=1}^{n}(-1)^{i+1}\mu(u_{n+1})\huaT\mu(u_i)\omega(u_1, \cdots,\hat{u_i},\cdots,u_n)\\
&&+\sum_{1\leq i<j\leq n}(-1)^{i+j}\mu(u_{n+1})\omega([u_i,u_j]_C,u_1,\cdots,\hat{u_i},\cdots,\hat{u_j},\cdots,u_{n}).
\end{eqnarray*}
By \eqref{rep3} and $u_{i}\cdot_T u_{n+1}=\rho(T(u_{i}))u_{n+1}$, we have
$\Xi(\partial\omega)=\dM(\Xi(\omega))$. Thus, $\Xi$ is a homomorphism between the cochain complexes, that induces a homomorphism $\Xi_*$ between the corresponding cohomologies.
\end{proof}
\section{Classification of abelian extensions of relative Rota-Baxter Lie algebras}\label{sec:a}

In this section, we classify abelian extensions of relative Rota-Baxter Lie algebras using the second cohomology group of a relative Rota-Baxter Lie algebra with coefficients in a representation.

\begin{defi}
Let $((\g,[\cdot,\cdot]_\g),(V,\rho),T)$ and $((\h,[\cdot,\cdot]_\h),(W,\nu),\huaT)$ be two relative Rota-Baxter Lie algebras. An {\bf extension} of $((\g,[\cdot,\cdot]_\g),(V,\rho),T)$ by $((\h,[\cdot,\cdot]_\h),(W,\nu),\huaT)$ is a short exact sequence of relative Rota-Baxter Lie algebra homomorphisms:
\[\begin{CD}
0@>>>W@>i>>\hat{V}@>p>>V             @>>>0\\
@.    @V \huaT VV   @V\hat{T}VV  @V T VV    @.\\
0@>>>\h @>\frak{i}>>\hat{\g}@>\frak{p}>>\g             @>>>0
,
\end{CD}\]
where $((\hat{\g},[\cdot,\cdot]_{\hat{\g}}),(\hat{V},\hat{\rho}),\hat{T})$ is a relative Rota-Baxter Lie algebra.

An extension of $((\g,[\cdot,\cdot]_{\g}),(V,\rho),T)$ by $((\h,[\cdot,\cdot]_{\h}),(W,\nu),\huaT)$ is called {\bf abelian} if $\h$ is an abelian Lie algebra, and $\nu$ is a trivial representation, i.e. $\nu(\alpha)\xi=0$, for all $\alpha\in\h, \xi\in W$.
\end{defi}
\begin{defi}
A {\bf section} of an extension $((\hat{\g},[\cdot,\cdot]_{\hat{\g}}),(\hat{V},\hat{\rho}),\hat{T})$ of a relative Rota Baxter Lie algebra $((\g,[\cdot,\cdot]_{\g}),(V,\rho),T)$ by $((\h,[\cdot,\cdot]_{\h}),(W,\nu),\huaT)$ consists of linear maps $\frak{s}:\g\rightarrow\hat{\g}$ and $ s:V\rightarrow\hat{V}$ such that
\begin{equation*}
\frak{p}\circ\frak{s}=\Id,\quad p\circ s=\Id.
\end{equation*}
\end{defi}
In the sequel, we only consider abelian extensions. Let $(\frak{s}, s)$ be a section. Define linear maps $\rho_{\h}:\g\rightarrow\gl(\h),~\rho_{W}:\g\rightarrow\gl(W),$ and $\mu: V\rightarrow\Hom(\h,W)$ by
\begin{eqnarray*}
\rho_{\h}(x)\alpha&=&[\frak{s}(x),\alpha]_{\hat{\g}},\\
\rho_{W}(x)\xi&=&\hat{\rho}(\frak{s}(x))\xi,\\
\mu(u)\alpha&=&-\hat{\rho}(\alpha)s(u), \quad \forall x\in\g, \xi\in W, u\in V, \alpha\in\h.
\end{eqnarray*}
Then we have the following result.
\begin{pro}\label{without section}
With the above notations, the triple $(\rho_\h,\rho_W,\mu)$ is a representation of the relative Rota-Baxter Lie algebra $((\g,[\cdot,\cdot]_\g),(V,\rho),T)$ on the $2$-term complex of vector spaces $W\stackrel{\huaT}\rightarrow\h$.
Moreover, this representation is independent on the choice of sections.
\end{pro}
\begin{proof}
Since $\hat{\rho}(\alpha)\xi=\nu(\alpha)\xi=0$, for all $\alpha\in\h, \xi\in W$, we have $\hat{\rho}(\frak{s}([x,y]_\g)-[\frak{s}(x),\frak{s}(y)]_{\hat{\g}})\xi=0$. Moreover, by the fact that $\hat{\rho}$ is a representation of the Lie algebra $\hat{\g}$ on the vector space $\hat{V}$, we have
\begin{eqnarray*}
\nonumber\rho_{W}([x,y]_\g)\xi&=&\hat{\rho}(\frak{s}([x,y]_\g))\xi\\
\nonumber&=&\hat{\rho}([\frak{s}(x),\frak{s}(y)]_{\hat{\g}})\xi+\hat{\rho}(\frak{s}([x,y]_\g)-[\frak{s}(x),\frak{s}(y)]_{\hat{\g}})\xi\\
\nonumber&=&[\hat{\rho}(\frak{s}(x)),\hat{\rho}(\frak{s}(y))]\xi\\
\label{rep1}&=&[\rho_{W}(x),\rho_{W}(y)]\xi.
\end{eqnarray*}
Thus $\rho_{W}$ is a representation of the Lie algebra $(\g, [\cdot,\cdot]_\g)$ on the vector space $W$.

Since $\h$ is abelian, we have
\begin{eqnarray*}
\nonumber\rho_\h([x,y]_\g)\alpha&=&[\frak{s}([x,y]_\g),\alpha]_{\hat{\g}}\\
\nonumber&=&[[\frak{s}(x),\frak{s}(y)]_{\hat{\g}}+\frak{s}([x,y]_\g)-[\frak{s}(x),\frak{s}(y)]_{\hat{\g}},\alpha]_{\hat{\g}}\\
\nonumber&=&[[\frak{s}(x),\frak{s}(y)]_{\hat{\g}},\alpha]_{\hat{\g}}\\
\nonumber&=&[[\frak{s}(x),\alpha]_{\hat{\g}},\frak{s}(y)]_{\hat{\g}}+[\frak{s}(x),[\frak{s}(y),\alpha]_{\hat{\g}}]_{\hat{\g}}\\
\label{rep2}&=&[\rho_\h(x),\rho_\h(y)]\alpha.
\end{eqnarray*}
Thus $\rho_\h$ is a representation of the Lie algebra $(\g, [\cdot,\cdot]_\g)$ on the vector space $\h$.

Furthermore, we have
\begin{eqnarray*}
\rho_{W}(x)\mu(u)\alpha-\mu(u)\rho_{\h}(x)\alpha&=&-\hat{\rho}(\frak{s}(x))\hat{\rho}(\alpha)s(u)+\hat{\rho}([\frak{s}(x),\alpha]_{\hat{\g}})s(u)\\
&=&-\hat{\rho}(\alpha)\hat{\rho}(\frak{s}(x))s(u).
\end{eqnarray*}
Since $(\frak{p},p)$ is a homomorphism from $((\hat{\g}, [\cdot,\cdot]_{\hat{\g}}), (\hat{V}, \hat{\rho}), \hat{T})$ to $((\g, [\cdot,\cdot]_\g), (V, \rho), T)$, we have
\begin{equation*}
\hat{\rho}(\frak{s}(x))s(u)-s(\rho(x)u)\in W.
\end{equation*}
Thus we have
\begin{eqnarray*}\label{Representation33}
\mu(\rho(x)u)\alpha&=&-\hat{\rho}(\alpha)s(\rho(x)u)\\
&=&-\hat{\rho}(\alpha)\hat{\rho}(\frak{s}(x))s(u)+\hat{\rho}(\alpha)(\hat{\rho}(\frak{s}(x))s(u)-s(\rho(x)u))\\
&=&-\hat{\rho}(\alpha)\hat{\rho}(\frak{s}(x))s(u),
\end{eqnarray*}
which implies that
\begin{equation}\label{repre4}
\mu(\rho(x)u)=\rho_{W}(x)\mu(u)-\mu(u)\rho_{\h}(x).
\end{equation}

 By $\frak{p}(\frak{s}T(u)-\hat{T}(s(u)))=T(u)-T(p(s(u)))=0,$ we have $(\frak{s}T(u)-\hat{T}(s(u)))\in\h.$ Since $((\hat{\g},[\cdot,\cdot]_{\hat{\g}}),(\hat{V},\hat{\rho}),\hat{T})$ is a relative Rota-Baxter Lie algebra, and $\h$ is an abelian Lie algebra, we have
\begin{eqnarray*}
[\hat{T}(s(u)),\huaT(\xi)]_{\hat{\g}}&=&\hat{T}(\hat{\rho}(\hat{T}(s(u)))\xi)-\hat{T}(\hat{\rho}(\huaT(\xi))s(u))\\
&=&\hat{T}((\hat{\rho}(\frak{s}(T(u)))-\hat{\rho}(\frak{s}T(u)+\hat{T}(s(u)))\xi)-\hat{T}((\hat{\rho}(\huaT(\xi))s(u))\\
&=&\hat{T}(\hat{\rho}(\frak{s}(T(u)))\xi)-\hat{T}(\hat{\rho}(\huaT(\xi))s(u)),
\end{eqnarray*}
and
\begin{equation*}
[\frak{s}T(u),\huaT(\xi)]_{\hat{\g}}=[\hat{T}(s(u))+\frak{s}(T(u))-\hat{T}(s(u)),\hat{T}(\xi)]_{\hat{\g}}=[\hat{T}(s(u)),\hat{T}(\xi)]_{\hat{\g}}.
\end{equation*}
Thus
\begin{eqnarray}
\nonumber\rho_\h(T(u))\huaT(\xi)&=&[\frak{s}(T(u)), \huaT(\xi)]_{\hat{\g}}\\
\nonumber&=&[\hat{T}(s(u)),\hat{T}(\xi)]_{\hat{\g}}\\
\nonumber&=&\hat{T}(\hat{\rho}(\frak{s}(T(u)))\xi)-\hat{T}((\hat{\rho}(\huaT(\xi))s(u))\\
\label{Representationrep3}
&=&\huaT(\rho_{W}(T(u))\xi)+\huaT(\mu(u)\huaT(\xi)).
\end{eqnarray}
 By  (\ref{repre4}) and \eqref{Representationrep3}, $(\rho_\h,\rho_W,\mu)$ is a representation of  $((\g,[\cdot,\cdot]_{\g}),(V,\rho),T)$ on $W\stackrel{\huaT}\rightarrow\h$.

Let $(\frak{s}',s')$ be another section, and  $(\rho'_\h,\rho'_W,\mu')$ be the corresponding representation of the relative Rota-Baxter Lie algebra $((\g,[\cdot,\cdot]_{\g}),(V,\rho),T)$ on $W\stackrel{\huaT}\rightarrow\h$. Since $\frak{s}'(x)-\frak{s}(x)\in\h$ and $\h$ is abelian, we have
\begin{equation*}
(\rho'_\h(x)-\rho_\h(x))\alpha= [\frak{s}'(x),\alpha]_{\hat{\g}}-[\frak{s}(x),\alpha]_{\hat{\g}}=[\frak{s}'(x)-\frak{s}(x),\alpha]_{\hat{\g}}=0.
\end{equation*}
Thus, we have $\rho'_\h(x)=\rho_\h(x).$ Similarly, we have $ \rho'_W(x)=\rho_W(x).$ By $s'(u)-s(u)\in W,$ we have $\mu'(u)=\mu(u)$. Thus the representation $(\rho_\h, \rho_W, \mu)$ is independent on the choice of sections.
\end{proof}
Let $(\frak{s}, s)$ be a section. We further define
$
\omega\in\Hom(\wedge^{2}\g,\h),~ \varpi\in\Hom(\g\otimes V,W),~ \chi\in\Hom(V,\h)
$
by
\begin{eqnarray*}
\omega(x,y)&=&[\frak{s}(x),\frak{s}(y)]_{\hat{\g}}-\frak{s}([x,y]_\g),\\
\varpi(x,u)&=&\hat{\rho}(\frak{s}(x))s(u)-s(\rho(x)u),\\
\chi(u)&=&\hat{T}(s(u))-\frak{s}(T(u)).
\end{eqnarray*}
Define $\huaS:\g\oplus \h\lon\hat{\g}$ and $S:V\oplus W\lon\hat{V}$ by
$$
\huaS(x+\alpha)=\frks(x)+\alpha,\quad S(u+\xi)=s(u)+\xi.
$$
It is obvious that $\huaS$ and $S$ are isomorphisms between vector spaces. Transfer the relative Rota-Baxter Lie algebra structure on $\hat{\g}$ and $\hat{V}$ to $\g\oplus\h$ and $V\oplus W$ via the isomorphisms $\huaS$ and $S$. We obtain a relative Rota-Baxter Lie algebra $((\g\oplus\h, [\cdot,\cdot]_\omega), (V\oplus W, \varrho), \frak{T})$, where $[\cdot,\cdot]_\omega, \varrho,$ and $\frak{T}$ are given by
\begin{eqnarray*}
[x_{1}+\alpha_{1},x_{2}+\alpha_{2}]_\omega&=&\huaS^{-1}[\huaS(x_{1}+\alpha_{1}),\huaS(x_{2}+\alpha_{2})]_{\hat{\g}}\\&=&[x_{1},x_{2}]_\g+[\frak{s}(x_{1}),\alpha_{2}]_{\hat{\g}}-[\frak{s}(x_{2}),\alpha_{1}]_{\hat{\g}}+[\frak{s}(x),\frak{s}(y)]_{\hat{\g}}-\frak{s}([x,y]_\g)\\
&=&[x_{1},x_{2}]_\g+\rho_{\h}(x_{1})\alpha_{2}-\rho_{\h}(x_2)\alpha_{1}+[\frak{s}(x),\frak{s}(y)]_{\hat{\g}}-\frak{s}([x,y]_\g)\\
&=&[x_{1},x_{2}]_\g+\rho_{\h}(x_{1})\alpha_{2}-\rho_{\h}(x_2)\alpha_{1}+\omega(x,y),\\
\varrho(x+\alpha)(u+\xi)&=&S^{-1}\hat{\rho}(\huaS(x+\alpha))S(u+\xi)\\&=&\rho(x)u+\hat{\rho}(\frak{s}(x))s(u)-s(\rho(x)u)+\hat{\rho}(\frak{s}(x))\xi+\hat{\rho}(\alpha)s(u)\\
&=&\rho(x)u+\rho_{W}(x)\xi-\mu(u)\alpha+\hat{\rho}(\frak{s}(x))s(u)-s(\rho(x)u)\\
&=&\rho(x)u+\rho_{W}(x)\xi-\mu(u)\alpha+\varpi(x,u),\\
\frak{T}(u+\xi)&=&\huaS^{-1}\hat{T}S(u+\xi)\\&=&T(u)+\huaT(\xi)+\hat{T}(s(u))-\frak{s}(T(u))\\
&=&T(u)+\huaT(\xi)+\chi(u).
\end{eqnarray*}

\begin{thm}\label{thm:cohomologicalclass}
  With the above notations, $(\omega,\varpi,\chi)$ is a $2$-cocycle of the relative Rota-Baxter Lie algebra $((\g,[\cdot,\cdot]_\g),(V;\rho),T)$  with coefficients in  $[W\stackrel{\huaT}{\lon}\h, \rho_\h, \rho_W, \mu]$. Moreover, its cohomological class does not depend on the choice of sections.
\end{thm}
\begin{proof}
 First by the fact that $[\cdot,\cdot]_{\omega}$ satisfies the Jacobi identity, we deduce that $\omega$ is $2$-cocycle of the Lie algebra $(\g,[\cdot,\cdot]_\g)$ with coefficients in $(\h,\rho_\h)$, i.e. $\dM_{\CE}\omega=0$.

 Since $\varrho$ is a representation of the Lie algebra  $(\g\oplus\h,[\cdot,\cdot]_{\omega})$ on $V\oplus W$, we obtain
 \begin{eqnarray*}
 0&=&\varrho([x+\alpha,y+\beta]_{\omega})(u+\xi)-[ \varrho(x+\alpha), \varrho(y+\beta)](u+\xi)\\
 &=&\varrho([x,y]_\g+\rho_{\h}(x)\beta-\rho_{\h}(y)\alpha+\omega(x,y))(u+\xi)\\
 &&-\varrho(x+\alpha)\varrho(y+\beta)(u+\xi)+\varrho(y+\beta)\varrho(x+\alpha)(u+\xi)\\
 &=&\rho([x,y]_\g)u+\rho_{W}([x,y]_\g)\xi-\mu(u)\rho_{\h}(x)\beta+\mu(u)\rho_{\h}(y)\alpha-\mu(u)\omega(x,y)+\varpi([x,y]_\g,u)\\
 &&-\Big(\rho(x)\rho(y)u+\rho_{W}(x)\rho_{W}(y)\xi-\rho_{W}(x)\mu(u)\beta+\rho_{W}(x)\varpi(y,u)-\mu(\rho(y)u)\alpha+\varpi(x,\rho(y)u)\\
 &&-\rho(y)\rho(x)u-\rho_{W}(y)\rho_{W}(x)\xi+\rho_{W}(y)\mu(u)\alpha-\rho_{W}(y)\varpi(x,u)+\mu(\rho(x)u)\beta-\varpi(y,\rho(x)u)\Big)\\
 &=&\varpi([x,y]_\g,u)-\mu(u)\omega(x,y)-\rho_{W}(x)\varpi(y,u)-\varpi(x,\rho(y)u)+\rho_{W}(y)\varpi(x,u)+\varpi(y,\rho(x)u),
 \end{eqnarray*}
which implies that $\delta(\omega,\varpi)=0.$

 We have
 \begin{eqnarray*}
 [\frak{T}(u+\xi),\frak{T}(v+\eta)]_\omega
 &=&[T(u),T(v)]_\g+\rho_{\h}(T(u))\huaT(\eta)+\rho_{\h}(T(u))\chi(v)-\rho_{\h}(T(v))\huaT(\xi)\\
 &&-\rho_{\h}(T(v))\chi(u)+\omega(T(u),T(v)),
 \end{eqnarray*}
and
\begin{eqnarray*}
&&\frak{T}(\varrho(\frak{T}(u+\xi))(v+\eta))-\frak{T}(\varrho(\frak{T}(v+\eta))(u+\xi))\\
&=&T(\rho(T(u))v)+\huaT(\rho_{W}(T(u))\eta)-\huaT(\mu(v)\huaT(\xi))-\huaT(\mu(v)\chi(u))+\huaT(\varpi(T(u),v))+\chi(\rho(T(u))v)\\
&&-T(\rho(T(v))u)-\huaT(\rho_{W}(T(v))\xi)+\huaT(\mu(u)\huaT(\eta))+\huaT(\mu(u)\chi(v))-\huaT(\varpi(T(v),u))-\chi(\rho(T(v))u).
\end{eqnarray*}
By the fact that $\frak{T}$ is a relative Rota-Baxter operator and
\begin{equation*}
\rho_{\h}(T(v))\huaT(\xi)=\huaT(\rho_{W}(T(v))\xi)+\huaT(\mu(v)\huaT(\xi)),
\end{equation*}
we have
\begin{eqnarray*}
0&=&\rho_\h(T(u))\chi(v)-\rho_\h(T(v))\chi(u)-\huaT(\mu(u)\chi(v))+\huaT(\mu(v)\chi(u))-\chi(\rho(T(u))v)\\
&&+\chi(\rho(T(v))u)+\omega(T(u),T(v))+\huaT(\varpi(T(v),u))-\huaT(\varpi(T(u),v))\\
&=&\partial\chi(u,v)+h_T(\omega,\varpi)(u,v),
\end{eqnarray*}
which implies that $ h_T(\omega,\varpi)+\partial\chi=0$. Therefore, $\huaD_R(\omega,\varpi,\chi)=0,$ i.e. $(\omega,\varpi,\chi)$ is a 2-cocycle.

Let $(\frak{s}',s')$ be another section and $(\omega',\varpi',\chi')$ be the associated  2-cocycle. Assume that $\frak{s}'=\frak{s}+N$ and $s'=s+S$ for $N\in\Hom(\g,\h)$ and $S\in\Hom(V,W)$. Then we have
\begin{eqnarray*}
  (\omega'-\omega)(x,y)&=&[\frak{s}'(x),\frak{s}'(y)]_{\hat{\g}}-\frak{s}'[x,y]_\g-[\frak{s}(x),\frak{s}(y)]_{\hat{\g}}+\frak{s}[x,y]_\g\\
  &=&-N([x,y]_\g)=\dM_{\CE} N(x,y),\\
  (\varpi'-\varpi)(x,u)&=&\hat{\rho}(\frak{s}'(x))s'(u)-s'(\rho(x)u)-\hat{\rho}(\frak{s}(x))s(u)+s(\rho(x)u)\\
  &=&\rho_{W}(x)(S(u))-\mu(u)N(x)-S(\rho(x)u),\\
  (\chi'-\chi)(u)&=&\hat{T}(s'(u))-\frak{s}'(T(u))-\hat{T}(s(u))+\frak{s}(T(u))\\
  &=&\huaT(S(u))-N(T(u)),
\end{eqnarray*}
which implies that $(\omega',\varpi',\chi')-(\omega,\varpi,\chi)=\huaD_R(N,S)$. Thus, $(\omega',\varpi',\chi')$ and $(\omega,\varpi,\chi)$ are in the same cohomology class.
\end{proof}

Isomorphisms between abelian extensions can be obviously defined as follows.
\begin{defi}
Let $((\hat{\g},[\cdot,\cdot]_{\hat{\g}}),(\hat{V},\hat{\rho}),\hat{T})$ and $((\tilde{\g},[\cdot,\cdot]_{\tilde{\g}}),(\tilde{V},\tilde{\rho}),\tilde{T})$ be two abelian extensions of a relative Rota-Baxter Lie algebra $((\g,[\cdot,\cdot]_\g),(V,\rho),T)$ by $(\h ,W ,\huaT)$. They are said to be {\bf isomorphic} if there exists an isomorphism $(\kappa,\lambda)$ of relative Rota-Baxter Lie algebras such that the following diagram commutes:
  \begin{equation*}
\xymatrix@!0{0\ar@{->} [rr]&& W \ar@{->} [rr] \ar'[d] [dd] \ar@{=} [rd] && \tilde{V}\ar'[d] [dd]\ar@ {->} [rr] \ar@{->} [rd]^{\lambda}&& V\ar@{=} [rd]\ar'[d] [dd]\ar@{->} [rr]&&0&\\
&0\ar@{->} [rr]&& W\ar@{->} [rr]\ar@{->} [dd]&&\hat{V}\ar@{->} [dd]\ar@{->} [rr]&&V\ar@ {->} [dd]\ar@{->} [rr]&&0\\
0\ar@{->} [rr]&&\h\ar'[r] [rr] \ar@{=} [rd]&&\tilde{\g}\ar@{->} [rd]^{\kappa}\ar'[r] [rr]&&\g\ar@{=} [rd] \ar'[r] [rr]&&0&\\
&0\ar@{->} [rr]&& \h\ar@{->} [rr]&&\hat{\g}\ar@{->} [rr]&&\g\ar@{->} [rr]&&0.}
\end{equation*}
\end{defi}
\begin{thm}\label{extension 2}
   For a given representation $[W\stackrel{\huaT}\rightarrow\h,\rho_{\h},\rho_{W},\mu]$ of a relative Rota-Baxter Lie algebra $((\g, [\cdot,\cdot]_\g), (V,\rho), T)$, abelian extensions of $((\g,[\cdot,\cdot]_\g),(V;\rho),T)$  by   $(\h ,W ,\huaT)$ are classified by the second cohomology group $\huaH^2(\g,\rho,T; \rho_\h, \rho_W, \mu)$.
\end{thm}
\begin{proof}
 Let $((\hat{\g},[\cdot,\cdot]_{\hat{\g}}),(\hat{V},\hat{\rho}),\hat{T})$ and $((\tilde{\g},[\cdot,\cdot]_{\tilde{\g}}),(\tilde{V},\tilde{\rho}),\tilde{T})$ be two isomorphic abelian extensions. Assume that $(\frak{s},s)$ is a section of $((\tilde{\g},[\cdot,\cdot]_{\tilde{\g}}),(\tilde{V},\tilde{\rho}),\tilde{T})$, and $(\tilde{\omega},\tilde{\varpi},\tilde{\chi})$ is the corresponding 2-cocycle.
 Define $(\frak{s}',s')$ by
 \begin{equation*}
 \frak{s}'=\kappa\circ\frak{s},\quad s'=\lambda\circ s.
 \end{equation*}
 Then, it is obvious that $(\frak{s}',s')$ is a section of $((\hat{\g},[\cdot,\cdot]_{\hat{\g}}),(\hat{V},\hat{\rho}),\hat{T})$. We denote by $(\hat{\omega},\hat{\varpi},\hat{\chi})$ the corresponding 2-cocycle. Then we have
 \begin{eqnarray*}
 \hat{\omega}(x,y)&=&[\frak{s}'(x),\frak{s}'(y)]_{\tilde{\g}}-\frak{s}'([x,y]_\g)\\
 &=&[\kappa(\frks(x)),\kappa(\frks(y))]_{\tilde{\g}}-\kappa(\frks[x,y]_\g)\\
 &=&\kappa([\frks(x), \frks(y)]_{\hat{\g}}- \frks[x,y]_\g)\\
 &=&\tilde{\omega}(x,y).
 \end{eqnarray*}
 Similarly, we have $$\tilde{\varpi}=\hat{\varpi},\quad \tilde{\chi}=\hat{\chi}.$$
 By Theorem \ref{thm:cohomologicalclass},  isomorphic abelian extensions give rise to the same cohomological class in $\huaH^2(\g,\rho,T; \rho_\h, \rho_W, \mu)$.

 For the converse part, we choose a 2-cocycle $(\omega,\varpi,\chi)$, and define a bracket on $\g\oplus\h$ by
 \begin{equation*}
 [(x,\alpha),(y,\beta)]_{\omega}=[x,y]_\g+\rho_\h(x)\beta-\rho_\h(y)\alpha+\omega(x,y),\quad \forall x,y\in\g, \alpha,\beta\in\h.
 \end{equation*}
 By $\dM_{\CE}\omega=0$, it is straightforward to deduce that $(\g\oplus\h,[\cdot,\cdot]_{\omega})$ is a Lie algebra.

 Since $\delta(\omega,\varpi)=0$, it is straightforward to deduce that
 \begin{equation*}
 \varrho(x,\alpha)(u,\xi)=\rho(x)u+\rho_W(x)\xi-\mu(u)\alpha+\varpi(x,u),\quad\forall x\in\g, \alpha\in\h, u\in V, \xi\in W.
 \end{equation*}
is a representation of $\g\oplus\h$ on $V\oplus W$.

 Define a linear map $\frak{T}:V\oplus W\rightarrow\g\oplus\h$ by
 \begin{equation*}
 \frak{T}(u,\xi)=T(u)+\huaT(\xi)+\chi(u),\quad\forall u\in V, \xi\in W.
 \end{equation*}
 Since $\partial\chi+h_T(\omega,\varpi)=0$, it is straightforward to deduce that $\frak{T}$ is a relative Rota-Baxter operator. Thus $((\g\oplus\h, [\cdot,\cdot]_\omega), (V\oplus W, \varrho), \frak{T})$ is a relative Rota-Baxter Lie algebra, which is an abelian extension of $((\g,[\cdot,\cdot]_\g), (V, \rho), T)$ by $(\h, W, \huaT)$. More precisely, we have the following commutative diagram:
 \[\begin{CD}
0@>>>W@>i>>V\oplus W@>p>>V             @>>>0\\
@.    @V \huaT VV   @V\frak{T}VV  @V T VV    @.\\
0@>>>\h @>\frak{i}>>\g\oplus\h@>\frak{p}>>\g             @>>>0
.
\end{CD}\]
We choose another 2-cocycle $(\omega',\varpi',\chi')$, such that $(\omega,\varpi,\chi)$ and $(\omega',\varpi',\chi')$ are in the same cohomology class, i.e.
\begin{equation*}
(\omega-\omega',\varpi-\varpi',\chi-\chi')=(\dM_{\CE}N, (\delta(N,S))_W, h_T(N,S)),
\end{equation*}
where $N\in\Hom(\g,\h), S\in\Hom(V,W)$. We denote the corresponding relative Rota-Baxter Lie algebra by $((\g\oplus\h,[\cdot,\cdot]_{\omega'}), (V\oplus W, \varrho'), \frak{T}')$.

Define linear maps $\kappa:\g\oplus\h\rightarrow\g\oplus\h$ and $\lambda:V\oplus W\rightarrow V\oplus W$ by
\begin{eqnarray*}
\kappa(x,\alpha)&=&(x,N(x)+\alpha),\\
\lambda(u,\xi)&=&(u,S(u)+\xi),
\end{eqnarray*}
for all $x\in\g, \alpha\in\h, u\in V, \xi\in W$. By $\omega-\omega'=\dM_{\CE}N$, we deduce that $\kappa$ is a Lie algebra isomorphism and $\lambda$ is an isomorphism of vector spaces.

Since
\begin{equation*}
\chi(u)-\chi'(u)=h_T(N,S)u=-N(T(u))+\huaT(S(u)),
\end{equation*}
we have $\frak{T}'(\lambda(u,\xi))=\kappa(\frak{T}(u,\xi))$.

Since $\varpi-\varpi'=(\delta(N,S))_W$, we deduce that
\begin{equation*}
\lambda(\varrho(x,\alpha)(u,\xi))=\varrho'(\kappa(x,\alpha))(\lambda(u,\xi)).
\end{equation*}
Moreover, it is straightforward to see that we have the following commutative diagram:
  \begin{equation*}
\xymatrix@!0{0\ar@{->} [rr]&& W \ar@{->} [rr] \ar'[d] [dd] \ar@{=} [rd] && V\oplus W\ar'[d] [dd]\ar@ {->} [rr] \ar@{->} [rd]^{\lambda}&& V\ar@{=} [rd]\ar'[d] [dd]\ar@{->} [rr]&&0&\\
&0\ar@{->} [rr]&& W\ar@{->} [rr]\ar@{->} [dd]&&V\oplus W\ar@{->} [dd]\ar@{->} [rr]&&V\ar@ {->} [dd]\ar@{->} [rr]&&0\\
0\ar@{->} [rr]&&\h\ar'[r] [rr] \ar@{=} [rd]&&\g\oplus\h\ar@{->} [rd]^{\kappa}\ar'[r] [rr]&&\g\ar@{=} [rd] \ar'[r] [rr]&&0&\\
&0\ar@{->} [rr]&& \h\ar@{->} [rr]&&\g\oplus\h\ar@{->} [rr]&&\g\ar@{->} [rr]&&0.}
\end{equation*}
Therefore, the two abelian extensions are isomorphic.
\end{proof}

\section{Classification of skeletal relative Rota-Baxter Lie 2-algebras}\label{sec:2}
In this section, first we recall the relative Rota-Baxter Lie 2-algebras introduced in \cite{Sheng} and then classify skeletal relative Rota-Baxter Lie 2-algebras using the third cohomology group of relative Rota-Baxter Lie algebras.

It is well-known that the category of Lie 2-algebras and the category of 2-term $L_\infty$-algebras are equivalent \cite{baez:2algebras}. Thus, when we say ``a Lie 2-algebra'', we mean a $2$-term $L_\infty$-algebra in the sequel.\vspace{2mm}
\begin{defi}
A Lie $2$-algebra structure on a graded vector space $\huaG=\g_0\oplus \g_1$ consists of the following data:
\begin{itemize}
\item[$\bullet$] a linear map $\g_1\stackrel{\dM}{\longrightarrow}\g_0,$

\item[$\bullet$] a skew-symmetric bilinear map $\frkl_2:\g_i\times \g_j\longrightarrow
\g_{i+j},~0\leq i+j\leq 1$,

\item[$\bullet$] a  skew-symmetric trilinear map $\frkl_3:\wedge^3 \g_0\longrightarrow
\g_1$,
   \end{itemize}
   such that for any $x_i,x,y,z\in \g_0$ and $\alpha,\beta\in \g_1$, the following equalities are satisfied:
\begin{itemize}
\item[$\rm(i)$] $\dM \frkl_2(x,\alpha)=\frkl_2(x,\dM \alpha),\quad \frkl_2(\dM \alpha,\beta)=\frkl_2(\alpha,\dM \beta),$
\item[$\rm(ii)$]$\dM \frkl_3(x,y,z)=\frkl_2(x,\frkl_2(y,z))+\frkl_2(y,\frkl_2(z,x))+\frkl_2(z,\frkl_2(x,y)),$
\item[$\rm(iii)$]$ \frkl_3(x,y,\dM \alpha)=\frkl_2(x,\frkl_2(y,\alpha))+\frkl_2(y,\frkl_2(\alpha,x))+\frkl_2(\alpha,\frkl_2(x,y)),$
\item[$\rm(iv)$]the Jacobiator identity:\begin{eqnarray*}
\sum_{i=1}^4(-1)^{i+1}\frkl_2(x_i,\frkl_3(x_1,\cdots,\widehat{x_i},\cdots,x_4))+\sum_{i<j}(-1)^{i+j}\frkl_3(\frkl_2(x_i,x_j),x_1,\cdots,\widehat{x_i},\cdots,\widehat{x_j},\cdots,x_4)=0.\end{eqnarray*}
   \end{itemize}
   \end{defi}
\begin{defi}\label{defi:Lie-2hom}
Let $\huaG=(\g_0,\g_1,\dM,\frkl_2,\frkl_3)$ and $\huaG'=(\g_0',\g_1',\dM',\frkl_2',\frkl_3')$ be Lie $2$-algebras. A
 homomorphism $F$ from $\huaG$ to $ \huaG'$ consists of:
 linear maps $F_0:\g_0\rightarrow \g_0',~F_1:\g_{1}\rightarrow \g_{1}'$
 and $\huaF_{2}:\g_{0}\wedge \g_0\rightarrow \g_{1}'$,
such that the following equalities hold for all $ x,y,z\in \g_{0},
\alpha\in \g_{1},$
\begin{itemize}
\item [$\rm(i)$] $F_0\circ\dM=\dM'\circ F_1$,
\item[$\rm(ii)$] $F_{0}(\frkl_2(x,y))-\frkl_2'(F_{0}(x),F_{0}(y))=\dM'\huaF_{2}(x,y),$
\item[$\rm(iii)$] $F_{1}\frkl_2(x,\alpha)-\frkl_2'(F_{0}(x),F_{1}(\alpha))=\huaF_{2}(x,\dM \alpha)$,
\item[$\rm(iv)$]
$\huaF_2(\frkl_2(x,y),z)+c.p.+F_1(\frkl_3(x,y,z))=\frkl_2'(F_0(x),\huaF_2(y,z))+c.p.+\frkl_3'(F_0(x),F_0(y),F_0(z))$.
\end{itemize}\end{defi}
Let $\huaV:V_1\stackrel{\tau}{\longrightarrow}V_0$ be a complex of
vector spaces. Define $\End^0_\tau(\huaV)$ by
$$
\End^0_\tau(\huaV)\triangleq\{(A_0,A_1)\in\gl(V_0)\oplus
\gl(V_1)|A_0\circ\tau=\tau\circ A_1\},
$$
and define $\End^1(\huaV)\triangleq \Hom(V_0,V_1)$. There is a
differential $\Delta:\End^1(\huaV)\longrightarrow \End^0_\tau(\huaV)$
given by
$$
\Delta(\phi)\triangleq(\tau\circ\phi, \phi\circ\tau)\quad\forall~\phi\in\End^1(\huaV),
$$
and a bracket operation $[\cdot,\cdot]$ given by the graded
commutator. More precisely,  for any $A=(A_0,A_1),B=(B_0,B_1)\in
\End^0_\tau(\huaV)$ and $\phi\in\End^1(\huaV)$, $[\cdot,\cdot]$ is
given by
\begin{eqnarray*}
  [A,B]=A\circ B-B\circ A=(A_0\circ B_0-B_0\circ A_0,A_1\circ B_1-B_1\circ
  A_1),
\end{eqnarray*}
and
\begin{equation}\label{representation}
  ~[A,\phi]=A\circ \phi-\phi\circ A=A_1\circ \phi-\phi\circ A_0.
\end{equation}
These two operations make $\End^1(\huaV)\xrightarrow{\Delta}
\End^0_\tau(\huaV)$ into a strict Lie 2-algebra, which we denote by $\End(\huaV)$. It plays
the same role as $\gl(V)$ for a vector space $V$.
\begin{defi}
A representation of a Lie 2-algebra $\huaG$ on $\huaV$ is a homomorphism $(\rho_0,\rho_1,\rho_2)$ from $\huaG$ to $\End(\huaV)$.
\end{defi}
\begin{ex}{\rm
Let $\huaG=(\g_0,\g_{1}, \dM, \frkl_2,\frkl_3)$ be a Lie 2-algebra. We define $\ad_0: \g_0\rightarrow\End_{\tau}^{0}(\g_1\stackrel{\dM}{\rightarrow}\g_0), ~\ad_1: \g_1\rightarrow\Hom(\g_0,\g_1)$ and $\ad_2: \wedge^2\g_0\rightarrow\Hom(\g_0,\g_1)$ by
\begin{eqnarray*}
\ad_0(x)y=\frkl_2(x,y),\quad
\ad_0(x)\alpha=\frkl_2(x,\alpha),\quad
\ad_1(\alpha)(x)=\frkl_2(\alpha,x),\quad
\ad_2(x,y)(z)=-\frkl_3(x,y,z),
\end{eqnarray*}
for all $x, y, z\in\g_0, \alpha\in\g_1$. Then $(\ad_0, \ad_1, \ad_2)$ is a representation of $\huaG$ on $\g_1\stackrel{\dM}{\rightarrow}\g_0$, which is called the {\bf adjoint representation} of the Lie 2-algebra $\huaG$.
}
\end{ex}
Let $\huaG=(\g_0,\g_{1}, \dM, \frkl_2,\frkl_3)$ be a Lie 2-algebra and $(\rho_0,\rho_1,\rho_2)$ be a representation of $\huaG$ on a 2-term complex of vector spaces $\huaV=V_1\stackrel{\tau}{\rightarrow}V_0$. The notion of a relative Rota-Baxter operator (also called an $\huaO$-operator) on a Lie $2$-algebra was given in \cite{Sheng} in the study of solutions of $2$-graded classical Yang-Baxter equation. Similar to the fact that a relative Rota-Baxter operator on a Lie algebra induces a pre-Lie algebra \cite{Bai}, a relative Rota-Baxter operator on a Lie 2-algebra   induces a pre-Lie 2-algebra.
\begin{defi}\label{defi:O-operator}
 {\rm (\cite{Sheng})} A triple $(T_0,T_1,T_2)$, where $T_0:V_0\longrightarrow \g_0,T_1:V_1\longrightarrow \g_1$  and $T_2:\wedge^2V_0\longrightarrow\g_1$ are linear maps satisfying $\dM\circ T_1=T_0\circ \tau$, is called {\bf a relative Rota-Baxter operator} on $\huaG$ with respect to the representation  $(\rho_0,\rho_1,\rho_2)$, if for all $u,v,v_i\in V_0$ and $\xi\in V_1$ the following conditions are satisfied:
  \begin{itemize}
    \item[\rm(i)] $T_0\big(\rho_0(T_0(u))v-\rho_0(T_0(v))u\big)-\frkl_2(T_0(u),T_0(v))=\dM T_2(u,v);$

    \item[\rm(ii)] $T_1\big(\rho_1(T_1(\xi))v-\rho_0(T_0(v))\xi\big)-\frkl_2(T_1(\xi),T_0(v))= T_2(\tau(\xi),v);$

    \item[\rm(iii)] $
           \frkl_2(T_0(v_1),T_2(v_2,v_3))+T_2\big(v_3,\rho_0(T_0(v_1))v_2-\rho_0(T_0(v_2))v_1\big)\\
           +T_1\big(\rho_1(T_2(v_2,v_3))v_1+\rho_2(T_0(v_2),T_0(v_3))v_1\big)+c.p.+\frkl_3(T_0(v_1),T_0(v_2),T_0(v_3))=0.
      $
    \end{itemize}

 A {\bf relative Rota-Baxter Lie 2-algebra} consists of a Lie 2-algebra $\huaG=(\g_0,\g_{1}, \dM, \frkl_2,\frkl_3)$, a representation $(\rho_0,\rho_1,\rho_2)$ of $\huaG$ on a 2-term complex of vector spaces $\huaV=V_1\stackrel{\tau}{\longrightarrow}V_0$ and a relative Rota-Baxter operator $\Theta=(T_0,T_1,T_2)$. We usually denote a relative Rota-Baxter Lie 2-algebra by $(\huaG, \huaV, \Theta)$. In particular, a relative Rota-Baxter Lie 2-algebra is said to be {\bf skeletal} if $\dM=0$ and $\tau=0$.
    \end{defi}

Let $(\huaG, \huaV, \Theta)$ be a skeletal relative Rota-Baxter Lie $2$-algebra. Then $\huaG$ is a skeletal Lie 2-algebra. Therefore, we have the following equalities
 \begin{eqnarray}
\label{Lie}0&=&\frkl_2(\frkl_2(x,y),z)+\frkl_2(\frkl_2(y,z),x)+\frkl_2(\frkl_2(z,x),y),\\
\label{Lierep}0&=&\frkl_2(\frkl_2(x,y),\alpha)+\frkl_2(\frkl_2(y,\alpha),x)+\frkl_2(\frkl_2(\alpha,x),y),
\end{eqnarray}
and
\begin{eqnarray}
\nonumber0&=&-\frkl_3(\frkl_2(x,y),z,w)-\frkl_3(\frkl_2(y,z),x,w)-\frkl_3(\frkl_2(z,w),x,y)-\frkl_3(\frkl_2(x,w),y,z)\\
\label{equation2}&&+\frkl_3(\frkl_2(y,w),x,z)+\frkl_3(\frkl_2(x,z),y,w)+\frkl_2(\frkl_3(x,y,z),w)\\
\nonumber&&-\frkl_2(\frkl_3(y,z,w),x)+\frkl_2(\frkl_3(x,z,w),y)-\frkl_2(\frkl_3(x,y,w),z),
\end{eqnarray}
for all $x, y, z, w\in\g_0, \alpha\in\g_1.$

Now we are ready to give the main result in this section.

\begin{thm}\label{theorem H3}
There is a one-to-one correspondence between skeletal relative Rota-Baxter Lie $2$-algebras and $3$-cocycles of   relative Rota-Baxter Lie algebras.
\end{thm}
\begin{proof}
Let $(\huaG, \huaV, \Theta)$ be a skeletal relative Rota-Baxter Lie $2$-algebra. Define linear maps $\rho:\g_0\rightarrow \gl(V_0)$ and $[\cdot,\cdot]_{\g_0}: {\g_0}\wedge{\g_0}\rightarrow{\g_0}$ by
\begin{equation*}
\rho(x)u=\rho_0(x)u, \quad [x,y]_{\g_0}=\frkl_2(x,y), \quad x, y\in\g_0, u\in\g_1.
\end{equation*}
Since $(\rho_0, \rho_1, \rho_2)$ is a representation of $\huaG$ on $V_1\stackrel{0}{\rightarrow}V_0$, we have
\begin{equation}
\label{repV}\rho(\frkl_2(x,y))u=\rho(x)\rho(y)u-\rho(y)\rho(x)u.
\end{equation}
By $\rm(i)$ in Definition \ref{defi:O-operator} with $\dM=0$, we have
\begin{equation}
\label{rbLg}\frkl_{2}(T_0(u),T_0(v))=T_0(\rho(T_0(u))v)-T_0(\rho(T_0(v))u).
\end{equation}
By \eqref{Lie}, \eqref{repV} and \eqref{rbLg}, $((\g_0, [\cdot,\cdot]_{\g_0}), (V_0, \rho), T_0)$ is a relative Rota-Baxter Lie algebra.

Define linear maps $\rho_{V_1}:\g_0\rightarrow\gl(V_1)$, $\rho_{\g_1}:\g_0\rightarrow \gl(\g_1)$, and
$\mu:V_0\rightarrow\Hom(\g_1, V_1)$ by
\begin{equation*}
\rho_{V_1}(x)\xi=\rho_0(x)\xi, \quad  \rho_{\g_1}(x)\alpha=\frkl_2(x,\alpha),\quad \mu(u)\alpha=-\rho_1(\alpha)u,\quad x\in\g_0, \alpha\in\g_1, \xi\in V_1, u\in V_0.
\end{equation*}
Since $(\rho_0, \rho_1, \rho_2)$ is a representation of $\huaG$ on $V_1\stackrel{0}{\rightarrow}V_0$, we have
\begin{eqnarray}
\label{repW}\rho_{V_1}(\frkl_2(x,y))\xi&=&\rho_{V_1}(x)\rho_{V_1}(y)\xi-\rho_{V_1}(y)\rho_{V_1}(x)\xi,\\
\label{repW3}\mu(u)\rho_{\g_1}(x)\alpha&=&\rho_{V_1}(x)\mu(u)\alpha-\mu(\rho(x)u)\alpha.
\end{eqnarray}
By $\rm(ii)$ in Definition \ref{defi:O-operator} with $\tau=0$, we have
\begin{equation}
\label{repW4}\rho_{\g_1}(T_0(u))T_1(\xi)=T_1(\rho_{V_1}(T_0(u))\xi)+T_1(\mu(u)T_1(\xi)).
\end{equation}
By \eqref{Lierep}, \eqref{repW}, \eqref{repW3} and \eqref{repW4}, $[V_1\stackrel{T_1}{\rightarrow}\g_1, \rho_{\g_1}, \rho_{V_1}, \mu]$ is a representation of $((\g_0, [\cdot,\cdot]_{\g_0}), (V_0, \rho), T_0)$.

Define linear maps $f_{\g_1}:\wedge^3\g_0\rightarrow\g_1, f_{V_1}:\wedge^2\g_0\otimes V_0\rightarrow V_1$ and $ \theta:\wedge^2V_0\rightarrow\g_1$ by
\begin{equation*}
f_{\g_1}(x,y,z)=\frkl_3(x,y,z), \quad f_{V_1}(x,y,u)=-\rho_2(x,y)u, \quad  \theta(u,v)=-T_2(u,v).
\end{equation*}
Since $(\rho_0,\rho_1,\rho_2)$ is a representation  $\huaG$ on $V_1\stackrel{0}{\rightarrow}V_0$, we have
\begin{eqnarray}
0&=&f_{V_1}([x,y]_\g,z,u)-f_{V_1}([x,z]_\g,y,u)+f_{V_1}([y,z]_\g,x,u)+\mu(u)f_{\g_1}(x,y,z)\label{equation3}\\
\nonumber&&-\rho_{V_1}(x)f_{V_1}(y,z,u)+\rho_{V_1}(y)f_{V_1}(x,z,u)-\rho_{V_1}(z)f_{V_1}(x,y,u)\\
\nonumber&&+f_{V_1}(y,z,\rho(x)u)-f_{V_1}(z,x,\rho(y)u)+f_{V_1}(x,y,\rho(z)u).
\end{eqnarray}
By $\rm(iii)$ in Definition \ref{defi:O-operator}, we have
\begin{eqnarray}
\nonumber0&=&T_1(f_{V_1}(T_0(v_2),T_0(v_3),v_1))-T_1(f_{V_1}(T_0(v_1),T_0(v_3),v_2))+T_1(f_{V_1}(T_0(v_1),T_0(v_2),v_3))\\
\label{equation1}&&-f_{\g_1}(T_0(v_1),T_0(v_2),T_0(v_3))+\rho_{\g_1}(T_0(v_1))\theta(v_2,v_3)-\rho_{\g_1}(T_0(v_2))\theta(v_1,v_3)\\
\nonumber&&+\rho_{\g_1}(T_0(v_3))\theta(v_1,v_2)-T_1(\mu(v_1)\theta(v_2,v_3))+T_1(\mu(v_2)\theta(v_1,v_3))\\
\nonumber&&-T_1(\mu(v_3)\theta(v_1,v_2))-\theta(\rho(T_0(v_1))v_2,v_3)+\theta(\rho(T_0(v_2))v_1,v_3)\\
\nonumber&&+\theta(\rho(T_0(v_1))v_3,v_2)-\theta(\rho(T_0(v_3))v_1,v_2)-\theta(\rho(T_0(v_2))v_3,v_1)+\theta(\rho(T_0(v_3))v_2,v_1).
\end{eqnarray}
Thus, by \eqref{equation2}, \eqref{equation3} and \eqref{equation1}, we deduce that $((f_{\g_1},f_{V_1}),\theta)$ is a $3$-cocycle of the relative Rota-Baxter Lie algebra $((\g_0, [\cdot,\cdot]_{\g_0}), (V_0, \rho), T_0)$ with coefficients in the representation $[V_1\stackrel{T_1}{\rightarrow}\g_1, \rho_{\g_1}, \rho_{V_1}, \mu]$.

The proof of other direction is similar. So the details will be omitted.
\end{proof}

\section{Representations and cohomologies of Rota-Baxter Lie algebras}\label{sec:RB}

 In this section, we adapt the previous general framework of relative Rota-Baxter Lie algebras to give   representations and cohomologies of Rota-Baxter Lie algebras. Applications of the second and the third cohomology groups will also be given to classify abelian extensions of Rota-Baxter Lie algebras and skeletal Rota-Baxter Lie 2-algebras.

\begin{defi}
A {\bf representation of a Rota-Baxter Lie algebra} $(\g,[\cdot,\cdot]_\g, T)$ on a vector space $W$ with respect to a linear transformation $\huaT\in\gl(W)$ is a representation $\rho_W$ of the Lie algebra $\g$ on $W$, satisfying
\begin{equation*}
\rho_{W}(T(x))\circ \huaT =\huaT\circ\rho_{W}(T(x)) +\huaT\circ \rho_{W}(x)\circ \huaT, \quad \forall x\in\g.
\end{equation*}

\end{defi}

Denote a representation by $[W; \huaT, \rho_W]$.

\begin{ex}{\rm
Let $(\g,[\cdot,\cdot]_{\g},T)$  be a  Rota-Baxter Lie algebra.
Then it is straightforward to see that $[\g;T,\ad]$ is a representation, which is called the {\bf adjoint representation} of $(\g,[\cdot,\cdot]_{\g}, T)$.
}
\end{ex}

Similar to Proposition \ref{pro:dualrep}, we have the following result.

\begin{pro}
 Let $[W; \huaT, \rho_W]$ be a representation of a  Rota-Baxter Lie algebra $(\g,[\cdot,\cdot]_{\g},T)$. Then $[W^*; -\huaT^*, \rho_W^*]$ is   also a representation of   $(\g,[\cdot,\cdot]_{\g},T)$, which is called the {\bf dual representation}.
\end{pro}
\begin{ex}{\rm
Let $(\g,[\cdot,\cdot]_{\g},T)$  be a  Rota-Baxter Lie algebra. Then $(\g^*;-T^*,\ad^*)$ is a representation of $(\g,[\cdot,\cdot]_{\g},T)$, which is called the {\bf coadjoint representation}.
}
\end{ex}
Similar to Proposition \ref{semiRB}, we also have the following semidirect product characterization of representations of Rota-Baxter Lie algebras.

\begin{pro}
Let $[W; \huaT, \rho_W]$ be a representation of a Rota-Baxter Lie algebra $(\g,[\cdot,\cdot]_{\g}, T)$. Then $(\g\oplus W, [\cdot,\cdot]_{\ltimes}, \frak{T})$ is a Rota-Baxter Lie algebra, where $[\cdot,\cdot]_{\ltimes}$ is the semidirect product Lie bracket given by
\begin{equation*}
[x+u,y+v]_{\ltimes}=[x,y]_\g+\rho_W(x)v-\rho_W(y)u,\quad \forall x,y\in\g, u,v\in W,
\end{equation*}
and $\frak{T}: \g\oplus W\rightarrow\g\oplus W$ is a linear map given by
\begin{equation*}
\frak{T}(x+u)=T(x)+\huaT(u),\quad \forall x\in\g, u\in W.
\end{equation*}
\end{pro}
\begin{rmk}
Let $(A,T)$ be a Rota-Baxter associative algebra, i.e. $A$ is an associative algebra and $T:A\lon A$ is a linear map satisfying
 $$
 T(x)T(y)=T(T(x)y+xT(y)),\quad \forall x,y\in A.
 $$In \cite{GouLin}, the authors defined a  left  Rota-Baxter module over $(A,T)$ to be a pair $(V,\huaT)$ with a left $A$-module $V$ and a linear map $\huaT:V\rightarrow V$ satisfying
\begin{equation*}
T(x)\huaT(u)=\huaT\big(x\huaT(u)+T(x)u\big),\quad \forall x\in A, u\in V.
\end{equation*}
It is obvious that a Rota-Baxter associative algebra gives rise to a Rota-Baxter Lie algebra $(A,[\cdot,\cdot],T)$, where the Lie bracket $[\cdot,\cdot]$ is the commutator Lie bracket. Furthermore, it is straightforward to deduce that  a left Rota-Baxter module $(V,\huaT)$ gives rise to a representation $\rho$ of the Rota-Baxter Lie algebra $(A,[\cdot,\cdot],T)$ on $V$ with respect to $\huaT:V\rightarrow V$, where $\rho(x)u=xu$, for all $x\in A, u\in V$. Thus the definition of representations of Rota-Baxter Lie algebras is consistent with representations of Rota-Baxter associative algebras given in \cite{GouLin}.
\end{rmk}

Let $[W; \huaT, \rho_W]$ be a representation of a Rota-Baxter Lie algebra $(\g,[\cdot,\cdot]_\g,T)$. Define the set of 0-cochains $\frak{C}^{0}(\g,T; \rho_W)$ to be 0, and define the set of $1$-cochains $\frak{C}^{1}(\g,T; \rho_W)$ to be $\Hom(\g,W)$. For $n\geq 2$, we define the set of $n$-cochains $\frak{C}^{n}(\g,T; \rho_W)$ by
\begin{equation*}
\frak{C}^{n}(\g,T; \rho_W)=\Hom(\wedge^n\g,W)\oplus\Hom(\wedge^{n-1}\g,W).
\end{equation*}
Define the coboundary operator
\begin{equation*}
\huaD_{\RB}:\frkC^n(\g,T; \rho_W)\lon \frkC^{n+1}(\g,T; \rho_W)
\end{equation*}
by
\begin{equation*}
\huaD_{\RB}(f,\theta)=(\dM_{\CE} f,\partial\theta+h_{T}(f)),\quad \forall f\in\Hom(\wedge^{n}\g,W), \theta\in\Hom(\wedge^{n-1}\g,W),
\end{equation*}
where
 \begin{itemize}
   \item  $\dM_{\CE}: \Hom(\wedge^{n}\g, W)\rightarrow\Hom(\wedge^{n+1}\g, W)$ is the Chevalley-Eilenberg coboundary operator of the Lie algebra $\g$ with coefficients in the representation $(W,\rho_W)$.
   \item  $h_{T}: \Hom(\wedge^{n}\g, W)\rightarrow\Hom(\wedge^{n}\g, W)$ is defined by
   \begin{eqnarray*}
   &&h_{T}(f)(x_{1},\cdots,x_{n})\\
   &=&(-1)^{n}f(T(x_{1}),T(x_{2}),\cdots,T(x_{n}))\\
   &&-(-1)^{n}\sum_{i=1}^{n}\huaT( f(T(x_{1}),\cdots,T(x_{i-1}),x_i,T(x_{i+1}),\cdots,T(x_{n}))).
   \end{eqnarray*}
   \item $\partial: \Hom(\wedge^{n-1}\g, W)\rightarrow\Hom(\wedge^{n}\g, W)$ is defined by
   \begin{eqnarray*}
&&\partial\theta(x_{1},\cdots,x_{n})\\
&=&\sum_{i=1}^{n}(-1)^{i+1}\rho_W(T(x_{i}))\theta(x_{1},\cdots,\widehat{x_{i}},\cdots,x_{n})-\sum_{i=1}^{n}(-1)^{i+1}\huaT\big(\rho_W(x_{i})\theta(x_{1},\cdots,\widehat{x_{i}},\cdots,x_{n})\big)\\
&&+\sum_{1\leq i<j\leq n}(-1)^{i+j}\theta([T(x_{i}),x_{j}]_\g-[T(x_{j}),x_{i}]_\g,x_{1},\cdots,\widehat{x_{i}},\cdots,\widehat{x_{j}},\cdots,x_{n}).
\end{eqnarray*}
\end{itemize}
\begin{thm}
  With the above notations,  $(\oplus _{n=0}^{+\infty}\frkC^n(\g,T; \rho_W),\huaD_{\RB})$ is a cochain complex, i.e. $$\huaD_{\RB}\circ \huaD_{\RB}=0.$$
\end{thm}
\begin{proof}
 We only give a sketch of the proof and leave details to readers. Let $[W; \huaT, \rho_W]$ be a representation of $(\g, [\cdot,\cdot]_\g, T)$. Then $[W\stackrel{\huaT}{\rightarrow} W; \rho_W, \rho_W, \rho_W]$ is a representation of the corresponding relative Rota-Baxter Lie algebra $((\g,[\cdot,\cdot]_\g), (\g,\ad),T)$.  Consider the cochain complex $$(\oplus_{n=0}^{+\infty}\frkC^n(\g,\ad,T; \rho_W, \rho_W, \rho_W ),\huaD_R),$$
  given in Theorem \ref{cohomology-of-LLT}.  It is straightforward to deduce that
  $(\oplus _{n=0}^{+\infty}\frkC^n(\g,T; \rho_W),\huaD_{\RB})$ is its subcomplex. Thus $\huaD_{\RB}\circ \huaD_{\RB}=0.$
\end{proof}
\begin{defi}
  The cohomology of the cochain complex  $(\oplus_{n=0}^{+\infty}\frkC^n(\g,T; \rho_W),\huaD_{\RB})$ is called {\bf the  cohomology of the Rota-Baxter Lie algebra} with coefficients in the representation $[W; \huaT, \rho_W]$. The corresponding $n$-th cohomology group is denoted by $\huaH^n(\g,T; \rho_W)$.
\end{defi}

 We consider the first cohomology group $\huaH^1(\g, T; \rho_{W})$. A linear map $f\in\Hom(\g, W)=\frkC^1(\g, T; \rho_W)$ is a $1$-cocycle if and only if
\begin{equation*}
\dM_{\CE}f=0, \quad f(T(x))=\huaT(f(x)), \quad \forall x\in\g.
\end{equation*}
Since $\frkC^0(\g,T; \rho_{W})=0$, there is no $1$-coboundaries. Thus we have
\begin{equation*}
\huaH^1(\g, T; \rho_{W})=\{f\in\Hom(\g,W)|f\circ T =\huaT\circ f, ~\dM_{\CE}f=0\}.
\end{equation*}

Recall  derivations on a Rota-Baxter Lie algebra from Definition \ref{defi:der}. Similar to Proposition \ref{pro:der}, we have the following characterization of derivations on a Rota-Baxter Lie algebra.
\begin{pro}
A  derivation on a  Rota-Baxter Lie algebra $(\g,[\cdot,\cdot]_{\g},T)$ is a $1$-cocycle on $(\g,[\cdot,\cdot]_{\g},T)$  with coefficients in the adjoint representation.
\end{pro}

Now we introduce the notion of abelian extensions of Rota-Baxter Lie algebras and briefly show that they are classified by the second cohomology group.

\begin{defi}
Let $(\g,[\cdot,\cdot]_\g, T)$ and $(\h,[\cdot,\cdot]_\h, \huaT)$ be two Rota-Baxter Lie algebras. An {\bf extension} of $(\g,[\cdot,\cdot]_\g ,T)$ by $(\h,[\cdot,\cdot]_\h, \huaT)$ is a short exact sequence of Rota-Baxter Lie algebra homomorphisms:
\[\begin{CD}
0@>>>\h@>\frak{i}>>\hat{\g}@>\frak{p}>>\g            @>>>0\\
@.    @V \huaT VV   @V\hat{T}VV  @V T VV    @.\\
0@>>>\h @>\frak{i}>>\hat{\g}@>\frak{p}>>\g             @>>>0
.
\end{CD}\]
An extension is called an {\bf abelian extension} if $\h$ is an abelian Lie algebra.
\end{defi}
\begin{defi}
A {\bf section} of an extension $(\hat{\g},[\cdot,\cdot]_{\hat{\g}}, \hat{T})$ of a Rota-Baxter Lie algebra $(\g,[\cdot,\cdot]_{\g}, T)$ by $(\h,[\cdot,\cdot]_{\h}, \huaT)$ is a linear map $\frak{s}:\g\rightarrow\hat{\g},$ such that
\begin{equation*}
\frak{p}\circ\frak{s}=\Id.
\end{equation*}
\end{defi}

\begin{defi}
Let $(\hat{\g},[\cdot,\cdot]_{\hat{\g}},\hat{T})$ and $(\tilde{\g},[\cdot,\cdot]_{\tilde{\g}},\tilde{T})$ be two abelian extensions of a Rota-Baxter Lie algebra $(\g,[\cdot,\cdot]_\g,T)$ by $(\h ,\huaT)$. They are said to be {\bf isomorphic} if there exists an isomorphism $\kappa$ of Rota-Baxter Lie algebras such that the following diagram commutes:
  \begin{equation*}
\xymatrix@!0{0\ar@{->} [rr]&& \h \ar@{->} [rr] \ar'[d] [dd] \ar@{=} [rd] && \tilde{\g}\ar'[d] [dd]\ar@ {->} [rr] \ar@{->} [rd]^{\kappa}&& \g\ar@{=} [rd]\ar'[d] [dd]\ar@{->} [rr]&&0&\\
&0\ar@{->} [rr]&& \h\ar@{->} [rr]\ar@{->} [dd]&&\hat{\g}\ar@{->} [dd]\ar@{->} [rr]&&\g\ar@ {->} [dd]\ar@{->} [rr]&&0\\
0\ar@{->} [rr]&&\h\ar'[r] [rr] \ar@{=} [rd]&&\tilde{\g}\ar@{->} [rd]^{\kappa}\ar'[r] [rr]&&\g\ar@{=} [rd] \ar'[r] [rr]&&0&\\
&0\ar@{->} [rr]&& \h\ar@{->} [rr]&&\hat{\g}\ar@{->} [rr]&&\g\ar@{->} [rr]&&0.}
\end{equation*}
\end{defi}

For an abelian extension of $(\g, [\cdot,\cdot]_\g, T)$ by $(\h, \huaT)$, choose a section $\frak{s}$ and define a linear map by
\begin{equation*}
 \rho_{\h}(x)(\alpha)=[\frak{s}(x),\alpha]_{\hat{\g}}, \quad \forall x\in\g, \alpha\in\h.
\end{equation*}
  Then $\rho_\h$ is a representation of $(\g, [\cdot,\cdot]_\g, T)$ on $\h$ with respect to $\huaT$. Moreover, similar to Proposition \ref{without section}, this representation is independent on the choice of sections. Define linear maps $\chi: \g\rightarrow\h$ by
  \begin{equation*}
  \chi(x)=\hat{T}(\frak{s}(x))-\frak{s}(T(x)), \quad\forall x\in\g,
  \end{equation*}
  and $\omega:\wedge^2\g\rightarrow\h$ by
   \begin{equation*}
   \omega(x,y)=[\frak{s}(x),\frak{s}(y)]_{\hat{\g}}-\frak{s}[x,y]_\g,\quad \forall x,y\in\g.
   \end{equation*}
   Then similar to Theorem \ref{thm:cohomologicalclass}, we have $\huaD_{\RB}(\omega,\chi)=0$ and $[(\omega,\chi)]\in\huaH^{2}(\g, T; \rho_\h)$ is independent on the choice of sections.

   Similar to Theorem \ref{extension 2}, we have
\begin{thm}
Given a representation $\rho_\h$ of a Rota-Baxter Lie algebra $(\g,[\cdot,\cdot]_\g, T)$ on $(\h ,\huaT)$, abelian extensions of the  Rota-Baxter Lie algebra $(\g,[\cdot,\cdot]_\g, T)$ by $(\h ,\huaT)$ are classified by the second cohomology group $\huaH^2(\g, T; \rho_\h)$.
\end{thm}
\begin{proof}
  We have proved that an abelian extension gives to a cohomological class in $\huaH^{2}(\g, T; \rho_\h)$. Similar to the proof of Theorem \ref{extension 2}, isomorphic abelian extensions give rise to the same cohomological class. For the converse part, we choose a $2$-cocycle $(\omega, \chi)$, and define a bracket operation on $\g\oplus\h$ by
   \begin{equation*}
   [x+\alpha, y+\beta]_\omega=[x,y]_\g+\rho_\h(x)\beta-\rho_\h(y)\alpha+\omega(x,y), \quad \forall x,y\in\g, \alpha,\beta\in\h.
   \end{equation*}
   By $\dM_{\CE}\omega=0$, $(\g\oplus\h, [\cdot,\cdot]_\omega)$ is a Lie algebra. Define a linear map $\frak{T}: V\oplus W\rightarrow \g\oplus\h$ by
   \begin{equation*}
   \frak{T}(u,\xi)=T(u)+\huaT(\xi)+\chi(u),\quad \forall u\in V, \xi\in W.
   \end{equation*}
   Since $\partial\chi+h_{T}(\omega)=0, \frak{T}$ is a Rota-Baxter operator on the Lie algebra $(\g\oplus\h,[\cdot,\cdot]_\omega)$. Moreover, let $(\omega', \chi')$ be another $2$-cocycle such that $(\omega,\chi)$ and $(\omega',\chi')$ are in the same cohomology class. Assume that $(\omega, \chi)-(\omega', \chi')=\huaD_\RB (N)$ for $N\in\Hom(\g, \h)$. Then $\kappa:\g\oplus \h\lon\g\oplus \h$ defined by $\kappa(x,\alpha)=(x, \alpha+N(x)),$ for any $x\in\g, \alpha\in\h$, is an isomorphism from $(\g\oplus\h, [\cdot,\cdot]_\omega, \frak{T})$ to $(\g\oplus\h, [\cdot,\cdot]_{\omega'}, \frak{T}')$. This finishes the proof.
\end{proof}

  At the end of this section, we introduce the notion of a Rota-Baxter Lie 2-algebra and briefly show that skeletal Rota-Baxter Lie 2-algebras are classified by $3$-cocycles of Rota-Baxter Lie algebras.
\begin{defi}
A {\bf Rota-Baxter Lie $2$-algebra} consists of a Lie $2$-algebra   $\huaG=(\g_0,\g_{1}, \dM, \frkl_2,\frkl_3)$  and a relative Rota-Baxter operator $\Theta=(T_0,T_1,T_2)$ on $\huaG$ with respect to the  adjoint representation $(\ad_0,\ad_1,\ad_2)$. More precisely, the following equalities hold:
  \begin{itemize}
    \item[\rm(i)] $T_0\big(\frkl_2(T_0(x),y)-\frkl_2(T_0(y),x)\big)-\frkl_2(T_0(x),T_0(y))=\dM T_2(x,y);$

    \item[\rm(ii)] $T_1\big(\frkl_2(T_1(\alpha),x)-\frkl_2(T_0(x),\alpha)\big)-\frkl_2(T_1(\alpha),T_0(x))= T_2(\dM(\alpha),x);$

    \item[\rm(iii)] $
           \frkl_2(T_0(x_1),T_2(x_2,x_3))+T_2\big(x_3,\frkl_2(T_0(x_1),x_2)-\frkl_2(T_0(x_2),x_1)\big)\\
           +T_1\big(\frkl_2(T_2(x_2,x_3),x_1)-\frkl_3(T_0(x_2),T_0(x_3),x_1)\big)+c.p.+\frkl_3(T_0(x_1),T_0(x_2),T_0(x_3))=0.
      $
    \end{itemize}

  We will denote a Rota-Baxter Lie $2$-algebra by $(\huaG, \Theta)$. A Rota-Baxter Lie $2$-algebra is said to be {\bf skeletal } if $\dM=0$.
\end{defi}

For a skeletal Rota-Baxter Lie $2$-algebra, $(\g_0, \frkl_2, T_0)$ is a Rota-Baxter Lie algebra and we have a representation $\rho_{\g_1}$ of $(\g_0, \frkl_2, T_0)$ on the vector space $\g_1$ with respect to $T_1$ given by $\rho_{\g_1}(x)(\alpha)=\frkl_2(x,\alpha)$. Define linear maps $f: \wedge^3\g_0\rightarrow\g_1$ and $\theta: \wedge^2\g_0\rightarrow\g_1$ by
\begin{eqnarray*}
f(x,y,z)&=&l_3(x,y,z),\\
\theta(x,y)&=&-T_2(x,y).
\end{eqnarray*}
Similar to Theorem \ref{theorem H3}, we have $\huaD_{\RB}(f, \theta)=0.$ Thus we obtain the following result.
\begin{thm}
There is a one-to-one correspondence between skeletal Rota-Baxter Lie $2$-algebras and $3$-cocycles of Rota-Baxter Lie algebras.
\end{thm}

 \end{document}